\documentclass[11pt]{article}
\usepackage{amssymb, authblk}
\usepackage{amsmath,bbm}
\usepackage{fullpage}
\usepackage{amsthm, hyperref}
\usepackage{graphicx}
\usepackage{verbatim}
\usepackage{algorithm}
\usepackage{algpseudocode}
\usepackage[dvipsnames]{xcolor}
\usepackage{tikz}
\usetikzlibrary{arrows.meta}
\usetikzlibrary{decorations.pathreplacing}

\usepackage{natbib}

\algnewcommand\algorithmicinput{\textbf{INPUT:}}
\algnewcommand\INPUT{\item[\algorithmicinput]}
\algnewcommand\algorithmicoutput{\textbf{OUTPUT:}}
\algnewcommand\OUTPUT{\item[\algorithmicoutput]}

\usepackage{hyperref}[]
\hypersetup{
    colorlinks=true,
    linkcolor=blue,
    filecolor=magenta,      
    urlcolor=cyan,
      citecolor=blue,
    }

\usepackage{cleveref}

\newtheorem{theorem}{Theorem}
\newtheorem{lemma}[theorem]{Lemma}
\newtheorem{proposition}[theorem]{Proposition}
\newtheorem{corollary}[theorem]{Corollary}

\newtheorem{definition}{Definition}
\newtheorem{remark}{Remark}
\newtheorem{assumption}{Assumption}

\allowdisplaybreaks

\DeclareMathOperator*{\esssup}{ess\,sup}
\DeclareMathOperator*{\loglog}{loglog}

\title{A Note on Online Change Point Detection}
\author[1]{Yi Yu}
\author[2]{Oscar Hernan Madrid Padilla}
\author[3]{Daren Wang}
\author[4]{Alessandro Rinaldo}
\affil[1]{Department of Statistics, University of Warwick}
\affil[2]{Department of Statistics, University California, Los Angeles}
\affil[3]{Department of Statistics, University of Chicago}
\affil[4]{Department of Statistics \& Data Science, Carnegie Mellon University}
\date{\today}

\begin{document}

\maketitle

\begin{abstract}
We investigate sequential change point estimation and detection in univariate nonparametric settings, where a stream of independent observations from sub-Gaussian distributions with a common variance factor and piecewise-constant but otherwise unknown means are collected. We develop a simple CUSUM-based methodology that provably control the probability of false alarms or the average run length while minimizing, in a minimax sense, the detection delay.  We allow for all the model parameters to vary in order to capture a broad range of levels of statistical hardness for the problem at hand.  We further show how our methodology is applicable to the case in which multiple change points are to be estimated sequentially. 

\end{abstract}

\section{Introduction}\label{sec-introduction}

We investigate what may arguably be regarded as one of the most basic nonparametric online change point settings, whereby a sequence of independent, univariate sub-Gaussian random variables with known variance factor and unknown piecewise-constant means are  observed sequentially. We seek to determine in an online manner, i.e.~each time we acquire a new observation, whether the data collected so far provide sufficient evidence to conclude that the mean of the distribution has changed at the present time or in the near past. The quality of any online procedure deployed for such a task is characterized based on the type I error, i.e.~the probability of incorrectly declaring that a change has taken place, and on the delay it incurs before correctly identifying a change point. Ideally, a good online procedure should guarantee a small false alarm probability while suffering only a minimal detection delay. It should be intuitively clear that these two features are at odds with each other: a procedure with a small probability of false alarm is likely to react slowly to even relatively big changes in the mean of the underlying distribution, thus producing a large detection delay. Vice versa, a methodology that is very sensitive to fluctuations in the data is unlikely to reliably discriminate between noise and signal and is thus prone to raising false alarms. In order to characterize this trade-off we keep track explicitly of all the parameters affecting the difficulty of the change point detection task, such  as the sub-Gaussian variance factor, the magnitude and time of the distributional change and the targeted type I error. Our goal is to develop an online procedure that provably works under the most unfavourable settings for which inference is just barely possible. 

We begin by formalizing the problem with a general assumption used throughout.

\begin{assumption}\label{assump-1}
Assume that $\{X_1, X_2, \ldots\}$ is a sequence of independent random variables with unknown means $\mathbb{E}(X_i) = f_i$, $i = 1, 2, \ldots$ and such that $\sup_{i = 1, 2, \ldots}\|X_i\|_{\psi_2} \leq \sigma$.
\end{assumption}

 We recall that for a random variable $X$, its  Orlicz-$\psi_2$-norm is defined as
	\[
		\|X\|_{\psi_2} = \inf\left\{t > 0: \, \mathbb{E}\{\exp(X^2/t^2)\} \leq 2\right\}.
	\]	

At first we consider the change point model for which the means of the observations change after collecting $\Delta$ - an unknown number of observations, by an unknown amount $\kappa$, as described in the next assumption. Extension to models allowing for multiple change points will be discussed in Section \ref{sec-multi}. 

\begin{assumption}\label{assump-2}
Assume that there exists a positive integer $\Delta \geq 1$ such that
	\[
		f_1 = \cdots = f_{\Delta} \neq f_{\Delta + 1} = f_{\Delta + 2} = \cdots .
	\]	
	In addition, let 
	\[
		\kappa = |f_{\Delta} - f_{\Delta + 1}|.
	\]
\end{assumption}

We will write the probability of any event with respect to any distribution consistent with Assumptions \ref{assump-1} and \ref{assump-2} as $\mathbb{P}_{\Delta}\{ \cdot \}$ and, similarly, we will use $\mathbb{E}_{\Delta}\{\cdot\}$ for the corresponding expectation.  With a slight abuse of notation, we describe the case in which the observations have constant means by setting $\Delta = \infty$ and will use the corresponding notation $\mathbb{P}_{\infty}\{ \cdot \}$ and~$\mathbb{E}_{\infty}\{\cdot\}$.  With this notation, we remark that the change point is $\Delta + 1$ and $\Delta$ is the sample size of the observations from the pre-change distribution.

Our main goal is to design an online procedure that is provably able to detect a change point soon after time $\Delta$ and with a controlled false alarm probability, denoted throughout as $\alpha$. In detail, an online change point detection procedure is an extended stopping time $\widehat{t}$ taking values in $\mathbb{N} \cup \{ \infty \}$ with respect to the natural filtration generated by the data. The false alarm probability of an online change point procedure is given as
\[
	\mathbb{P}_{\Delta} \left\{ \widehat{t} \leq \Delta\right\}  \text{ for any } \Delta < \infty \quad \text{ and } \quad  \mathbb{P}_{\infty} \left\{  \widehat{t} < \infty \right\} \text{ otherwise,}
\]
and the detection delay is the random variable
\[
\left(\widehat{t} - \Delta\right)_+,
\]
which is only defined provided that $\Delta < \infty$. We will develop procedures that guarantee that (i) $\mathbb{P}_{\infty} \left\{  \widehat{t} < \infty \right\} \leq \alpha$, for a user-defined target false alarm probability $\alpha$ or $\mathbb{E}_\infty(\widehat{t}) \geq \gamma$ for a user-defined average run length value $\gamma$, and (ii) that, 	at the same time, for all $\Delta < \infty$, $\left(\widehat{t} - \Delta\right)_+$ is minimal.

The setting described in Assumptions~\ref{assump-1} and \ref{assump-2} allows one to completely characterize the hardness of the problem -- measured both by the false alarm probability and the detection delay -- as a function of the  upper bound $\sigma$ on the fluctuations, the target probability of false alarm $\alpha$, the pre-change sample size $\Delta$ and the jump size $\kappa$.  Intuitively, the difficulty of the change point detection task is increasing in $\sigma$ and decreasing in $\Delta$, $\alpha$ and $\kappa$. To formally capture this important aspect of the problem, we implicitly assume a sequence of change point models with respect to which the data may have originated. Accordingly, the parameters defining the statistical task at hand, namely the quadruple $(\Delta,\alpha, \sigma^2,\kappa)$, are not fixed but should be instead viewed as sequences, expressing a spectrum of levels of difficulty of the problem we are interested in. 

We make the following contributions:
\begin{itemize}
\item We develop a CUSUM-based procedure with a false alarm probability control, yielding a detection delay of order 
\begin{equation}\label{eq:rate.intro}
\sigma^2\kappa^{-2} \log(\Delta/\alpha)
\end{equation}
 with probability at least $ 1- \alpha$, for all $\Delta < \infty$ satisfying the signal-to-noise ratio condition
\begin{equation}\label{eq:snr.intro}
\Delta \kappa^2 \sigma^{-2} \succeq \log(\Delta/\alpha).
\end{equation}
Interestingly, the above expression matches exactly the signal-to-noise ratio quantity for the offline version of the same change point detection problem, as demonstrated for example in \cite{wang2018univariate}. We elaborate further on this connection later on in \Cref{sec-offline}.
\item We show that a straightforward  modification of our procedure guarantees that, for any target average run length $\gamma \geq \Delta$, $\mathbb{E}_{\infty}(\widehat{t}) \geq \gamma$ and, at the same time, with probability $ 1 - \gamma^{-1}$ the detection delay is of order $\sigma^2 \kappa^{-2} \log(\gamma)$, for all $\Delta \leq \gamma$ such that $\Delta \kappa^2 \sigma^{-2} \succeq \log(\gamma)$.
\item We construct lower bounds indicating that the magnitude of the detection delay \eqref{eq:rate.intro} is essentially nearly optimal, in a minimax sense, save for a logarithmic term, whenever the signal-to-noise ratio condition \eqref{eq:snr.intro} is in effect. 
\item We generalize  our procedure to the case of multiple change points and show that, in this setting, the signal-to-noise condition \eqref{eq:snr.intro} is in fact necessary for online change point localization. 
\item We discuss variants of our methodology that incur smaller computational costs. \cite{chen2020high} mentioned that the computational cost of an online procedure should be of linear order of the number of time points.  We would like to emphasize that this claim holds when both the before and after change point distributions are exactly known, which is essentially the situation discussed in \cite{chen2009kernel}.  On the contrary, in our paper, we deal with sub-Gaussian distributions with unknown means.  In this situation, we are not aware of nor expect to see any theoretically-justified methods with linear order computational costs.
\end{itemize}

The paper is organized as follows. Section \ref{sec:relavant_work}  provides a review of related work on online and offline change point detection.  The main algorithms are then presented in Section \ref{sec-upperbounds}. In particular, Section \ref{sec-delay} contains the theoretical guarantees of the main algorithm, Section \ref{sec-variants} discusses the variants, and some practical implementation aspects are investigated in Section \ref{sec-practical}. Extensions to multiple change points settings are the subject of Section \ref{sec-multi}.  Section \ref{sec-lowerbounds}  then provides a lower bound on the detection delay  which shows that our algorithm is nearly optimal.  The paper concludes with a discussion in Section \ref{sec-discussions}.  All the proofs are deferred to the Appendices.

\subsection{Relevant literature}	
\label{sec:relavant_work}

\cite{Wald1945}, as a prelude of the sequential analysis, kicked off the statistical research on online change point detection problems.  A famous extension of \cite{Wald1945} is the CUSUM statistic proposed in \cite{Page1954}.  The optimality of \cite{Wald1945} and \cite{Page1954} was, to the best of our knowledge, studied first in \cite{lorden1971procedures}, which showed that among all the estimators which have average run lengths lower bounded by $\gamma$, the optimal detection delay rate is of order $\log(\gamma)/\mathrm{KL}(F_0, F_1)$, as $\gamma \to \infty$, where $\mathrm{KL}(\cdot, \cdot)$ is the Kullback--Leiber divergence.  \cite{moustakides1986optimal} and \cite{ritov1990decision} reiterated this minimax result and showed that in the optimality framework studied in \cite{lorden1971procedures}, the CUSUM statistic is optimal.  Similar results have also been derived in the same framework using a change-of-measure argument under more general assumptions in \cite{lai1981asymptotic}, \cite{lai1998information} and \cite{lai2001sequential}, among others.  In almost all of the second half of the 20th century, the research on online change point detection focused on optimizing the expected delay time.  The motivations back then were mainly from the manufactory sector, with quality control as the centre applications.  The data type studied were mainly univariate sequences, and the results were almost all asymptotic.  We refer readers to \cite{lai1995sequential, lai2001sequential} for comprehensive reviews.  

Before proceeding, we would like to emphasize that there is a fundamental difference between the optimality results derived in the aforementioned work and the ones developed by us in \Cref{sec-lowerbounds}.  In short, when considering the minimax lower bounds, the previous work only allows the change point location to vary and the optimality is derived in an asymptotic sense, by letting the lower bound of the average run length diverge.  In this paper, we let all model parameters vary with the location of the change point when deriving the minimax lower bounds, and allow for fixed sample arguments.

The second act of online change point research was kicked off by \cite{chu1996monitoring}, who formally stated the existence of `noncontamination' data that one has a training data set of size $m$, i.e.~$X_i \sim F_0$, $i = 1, \ldots, m$.  The theoretical results built upon the above assumption are asymptotic in the sense that one lets $m$ grow unbounded.  One can control the type I error with this noncontamination condition.  Since \cite{chu1996monitoring}, a large number of papers have been produced in this line of work, including univariate mean change \citep[e.g.][]{aue2004delay, kirch2008bootstrapping}, linear regression coefficients change \citep[e.g.][]{aue2009delay, huvskova2012bootstrapping}, multivariate mean and/or variance change \citep[e.g.][]{mei2010efficient}, univariate nonparametric change \citep[e.g.][]{huvskova2010fourier, hlavka2016bootstrap, desobry2005online}, Bayesian online change point detection \citep[e.g.][]{fearnhead2007line}, to name but a few.   More recent work includes \cite{he2018sequential}, which studied the sequential change point detection in a sequence of random graphs.  \cite{kirch2018modified} used estimating equations as a unified framework to include the location shift, linear regression and autoregressive online change point detection.  \cite{kurt2018real} converted different high-dimensional and/or nonparametric data to a univariate statistic and used geometric entropy minimisation methods to define an acceptance region.  \cite{chen2019sequential} constructed similarity measures via $K$-nearest neighbour estimators and then proposed a counting-based statistic to conduct sequential change point detection.   \cite{dette2019likelihood} proposed a general framework for sequential change point detection case and obtained a limiting distribution of the proposed statistics.  The framework can be used to handle high-dimensional and nonparametric cases.  \cite{gosmann2019new} exploited a likelihood ratio based method and shared similar core techniques with \cite{dette2019likelihood}.   \cite{keshavarz2018sequential} considered online change point detection in a sequence of Gaussian graph models and obtained asymptotic type I and II error controls.  \cite{chen2020high} considered online change point detection in a sequence of Gaussian random vectors where the mean changes over time.  Comprehensive monographs and survey papers include \cite{siegmund2013sequential}, \cite{tartakovsky2014sequential} and \cite{namoano2019online}.

We would like to mention \cite{maillard2019sequential}.  This paper is the most relevant paper to ours, to the best of our knowledge, and has heavily inspired our paper.  \cite{maillard2019sequential} studied a univariate online mean change point detection problem and deployed the Laplace transform to control the probabilities of the events, on which the fluctuations are contained within desirable ranges, although the arguments thereof remain doubtful.  Given the large probability events established based on the Laplace transforms, type I error controls and large probability detection delays were studied.  A claim on the phase transition and robust analysis under multiple change points scenario were also available in \cite{maillard2019sequential}.  There are a number of differences between this paper and \cite{maillard2019sequential}: (i) Instead of using the Laplace transform to establish large probability events, we summon the concentration inequalities for sub-Gaussian random variables, the union bound results and the peeling arguments.  It was pointed out in the Discussion in \cite{maillard2019sequential}, other more advanced tools including the peeling arguments may improve the results by changing logarithmic terms to iterative logarithmic terms. We are, however, skeptical about the feasibility of such claim.  (ii)  In addition to the type I errors, which are available in \cite{maillard2019sequential}, we also provide average run length results and a parallel set of results by setting a lower bound for the average run lengths.  This is a common practice in applications and is widely used in the existing literature \citep[e.g.][]{lai1981asymptotic}.  (iii)  Despite that \cite{maillard2019sequential} has much more modern arguments than papers in the 20th century, the results are presented in a more restrictive way.  For instance, the `phase transition' and `detectability' are presented as a property of the location of the change point only.  In this paper, we will exploit a signal-to-noise ratio, which is a function of the jump size, the variance and the change point location jointly.  This setup enables further studies of high-dimensional data problems.

An important aspect of our problem that sets it aside from many of the contributions in the area is the fact that our procedure is not sensitive to the values of the pre- and post-change means, but only to their difference. Formally, we consider composite null and alternative hypotheses for the pre- and post-change distributions, respectively. This feature, combined with the nonparametric nature of our sub-Gaussianity assumption, which is not enough to specify likelihood functions, poses additional and significant challenges compared to the simpler case in which the pre- and post-change  distributions are known. From a statistical standpoint, it prevents us from relying on likelihood-ratio-type procedures, which are analytically simpler and have optimality properties. Computationally, it requires us to re-evaluate our test statistic over past observations, adding to the computational burden and thus making the design of sequential procedures challenging. 

\section{Detection delay and type I error controls} \label{sec-upperbounds}

We are going to consider a simple online procure, described in \Cref{alg-online},  that keeps track of the running maximum of the standard CUSUM statistics, formally defined next, and declares that a change point has occurred as soon as its value exceeds a pre-specified, time-dependent threshold.  Later in the paper, we will discuss different variants of this simple method, but they all share this basic core principle. 

\begin{definition}[CUSUM statistic] \label{def-cusum}
Given a sequence $\{X_t\}_{t = 1, 2, \ldots} \subset \mathbb{R}$ and a pair of integers $1 \leq s < t$, we define the CUSUM statistic and its population version as
	\[
		\widehat{D}_{s, t} = \left|\left(\frac{t-s}{ts}\right)^{1/2}\sum_{l = 1}^s X_l - \left\{\frac{s}{t(t-s)}\right\}^{1/2}\sum_{l = s+1}^t X_l\right| 
	\]
	and
	\[
		D_{s, t} = \left|\left(\frac{t-s}{ts}\right)^{1/2}\sum_{l = 1}^s f_l - \left\{\frac{s}{t(t-s)}\right\}^{1/2}\sum_{l = s+1}^t f_l\right|,
	\]
	respectively.  
\end{definition}

\begin{algorithm}[!h]
	\begin{algorithmic}
		\INPUT $\{X_u\}_{u=1, 2, \ldots} \subset \mathbb{R}$, $\{b_{u}, u = 2, 3, \ldots\} \subset \mathbb{R}$.
		\State $t \leftarrow 1$
		\While{$\mathrm{FLAG} = 0$}  
			\State{$t \leftarrow t + 1$;}
			\State{$\mathrm{FLAG} = 1 - \prod_{s = 1}^{t-1} 1\left\{\widehat{D}_{s, t} \leq b_{t}\right\}$;}
		\EndWhile
		\OUTPUT $t$.
		\caption{Online change point detection via CUSUM statistics} \label{alg-online}
	\end{algorithmic}
\end{algorithm}

\Cref{alg-online} scans through the data sequence using the CUSUM statistic and a sequence of pre-specified positive threshold values.  For any time point $t \geq 2$, as long as there exists an integer $s \in [1, t)$, such that the corresponding CUSUM statistic $\widehat{D}_{s, t}$ exceeds the threshold $b_{t}$, we decide that a change in mean has occurred prior to the current time point $t$.  \Cref{alg-online} is written in the way that it will never terminate if there is no change point declared.  In practice, \Cref{alg-online} is terminated based on a stopping rule decided by the user and our theory accommodates this.

In the rest of this section, we will first focus on one version of the detection procedure, providing its detection delay and type I error control in \Cref{sec-delay}.  We will then discuss two other common alternative procedures, together with their performances and connections among all three procedures in \Cref{sec-variants}.  To conclude this section, we will discuss some practical issues in \Cref{sec-practical}.

\subsection{Anlysis of \Cref{alg-online}}\label{sec-delay}

In this section, we analyze the performance of \Cref{alg-online}. Here and in the rest of the paper, we adopt a high-dimensional framework of reference, whereby we implicitly assume a sequence of change point models with respect to which the data may have originated. Accordingly, the parameters defining the statistical task at hand, namely the pre-change sample size $\Delta$ when it is finite, the variance parameter $\sigma$, the magnitude $\kappa$ of the change and the false alarm probability $\alpha$ are not fixed but are instead allowed to vary, in order to express a spectrum of levels of difficulty of the problem we are interested in. This is of course a convenient mathematical formalization, and is not intended to represent any real-life situation.  Since our theoretical guarantees will be formulated based on finite sample bounds, this formalism does not pose any issue.  

We begin by specifying the signal-to-noise ratio condition that we will require in our analysis.

\begin{assumption}\label{assump-snr}
There exists a sufficiently large absolute constant $C_{\mathrm{SNR}} > 0$ such that 
	\[
		\Delta \kappa^2 \sigma^{-2} \geq C_{\mathrm{SNR}} \log(\Delta/\alpha).
	\]	
\end{assumption}

 It is worth emphasizing that the above condition should be interpreted as a constraint that is expected to be satisfied for all values of the parameters $(\Delta,\sigma^2,\kappa, \alpha)$ simultaneously. For example, assume the type I error level $\alpha$ and the jump size parameter $\kappa$ to be fixed. Then, the above condition will constrain how large the variance parameter $\sigma^2$ is allowed to be in relation to $\Delta$ and, vice-versa, how small $\Delta$ can be with respect to $\sigma^2$. We will refer to any relationship that holds among $(\Delta,\sigma^2,\kappa, \alpha)$, such as the one expressed in Assumption~\ref{assump-snr}, as a scaling.

In our main result, we show that using, the CUSUM statistic and under the signal-to-noise ratio condition in Assumption~\ref{assump-snr}, \Cref{alg-online} is able to detect the change point with a detection delay of order up to $\sigma^2\kappa^{-2}\log(\Delta/\alpha)$, with probability at least $1 - \alpha$, using time-varying thresholds of order $\sigma \log^{1/2}(t/\alpha)$.  This rate is nearly minimax optimal, aside from a $\log(\Delta)$ factor, as we show in \Cref{sec-lowerbounds}.  

\begin{theorem}\label{thm-2}
Consider the settings described in Assumption~\ref{assump-1}. Let $\alpha \in (0, 1)$ and $\widehat{t}$ be the stopping time returned by \Cref{alg-online} with inputs $\{ X_t \}_{t=1,2,\ldots}$ and $\{ b_t \}_{t=2,3,\ldots}$, where
	\begin{equation}\label{eq-b-2}
			b_{t} = 2^{3/2}\sigma\log^{1/2}(t/\alpha).
		\end{equation}
		 If $\Delta = \infty$, then
	\begin{equation}\label{eq:thm-2.delta.infty}
		\mathbb{P}_{\infty} \left\{ \widehat{t} < \infty \right\} < \alpha.
		\end{equation}
Under Assumption \ref{assump-2}, we have
\begin{equation}\label{eq:thm-2.any.delta}
	\mathbb{P}_{\Delta} \left\{ \widehat{t} \leq  \Delta \right\} < \alpha,
\end{equation}
for any $\Delta \geq 1$. 
		If Assumptions \ref{assump-2} and \ref{assump-snr} both hold, then 
		\begin{equation}\label{eq:thm-2.main}
	\mathbb{P}_{\Delta} \left\{ \Delta < \widehat{t}  \leq  \Delta + C_d\frac{\sigma^2 \log(\Delta/\alpha)}{\kappa^2}\right\} \geq 1 - \alpha,
	\end{equation}
	where $C_d > 0$ is an absolute constant with $C_d < C_{\mathrm{SNR}}$.
\end{theorem}

\begin{remark}
In relation to the procedure described in \cite{lorden1971procedures}, a direct consequence of combining \Cref{thm-2} above and Theorem~2 in \cite{lorden1971procedures} is the following.  For $k \in \{1, 2, \ldots,\}$, let $\widehat{t}_k$ be the output of applying \Cref{alg-online} to $\{X_k, X_{k+1}, \ldots\}$ and let $\widehat{t}_* = \min\{\widehat{t}_k + k- 1, \, k = 1, 2, \ldots\}$.  We have that the average run length of $\widehat{t}_*$ satisfies $\mathbb{E}_0 (\widehat{t}_*) \geq 1/\alpha$.
\end{remark}

The proof of \Cref{thm-2} is given in \Cref{sec-proofs-theorems}. In our proofs we have not optimized the constants, which we conjecture can be made smaller through  a more refined analysis. As for the tuning parameter $b_t$ used by the procedure, we have made the multiplicative constant $2^{3/2}$ explicit, though any number larger than 2 can be used: see \eqref{eq-C-1-cond}.  In practice, since usually the variance parameter $\sigma$ is unknown, we recommend calibrating the values of $b_t$ via simulations, assuming the availability of data from a pre-change distribution, as it is often the case in applications. 

The guarantee \eqref{eq:thm-2.delta.infty} implies that, in the absence of any change points, the procedure will continue indefinitely with probability at least $ 1- \alpha$. On the other hand, when a change point does exist, as assumed in Assumption \ref{assump-2}, the bound \eqref{eq:thm-2.any.delta} guarantees that the false alarm probability of our procedure is no larger than $\alpha$, regardless of the actual value of the change point. 
Additionally, under the scaling on the signal-to-noise ratio given in Assumption \ref{assump-snr}, \eqref{eq:thm-2.main} provides a high-probability bound on the detection delay of order $\sigma^2\kappa^{-2} \log(\Delta/\alpha)$, which is increasing in $\sigma$ and $\Delta$ and decreasing in $\alpha$ and $\kappa$. In particular, combining \eqref{eq:thm-2.main} with Assumption~\ref{assump-snr}, we can see that the smaller $\alpha$ --  the tolerance on the type I error -- is,  the larger the required signal-to-noise ratio and the detection delay are.
As suggested by the minimax lower bound given below in Proposition~\ref{prop-lb}, our detection delay bound appears to be optimal. We refer readers to \Cref{sec-lowerbounds} for a detailed discussion of the optimality of our results and their relation to the existing literature. Finally, we remark the expression of detection delay bound is nearly identical to the one of the localization error found by \cite{wang2018univariate} for the problem of estimating the change points in offline settings, which too is nearly minimax rate-optimal. While this result may not be surprising, it is far from obvious. 

It is important to point out how the statement in \Cref{thm-2} implies that our algorithm, regarded as a testing procedure to detect the presence of a change point against the null hypothesis of no change point, has power at least $ 1 - \alpha$ under the signal to noise condition of Assumption~\ref{assump-2}. What is more, inspection of the proof shows that, should there be a change point at time $\Delta + 1$, the power of our CUSUM-statistic-based procedure is maximal immediately after $\Delta$ and degrades as times goes by. In particular, for any for all fixed $\Delta$, $\mathbb{P}_{\Delta}(\widehat{t}  < \infty) < 1$, i.e. the power of our procedure is non-trivial as we are not guaranteed that we will eventually detect a change point, no matter how many samples we acquire. This may be different in some existing literature \citep[e.g.][]{shin2020nonparametric}, but this is due to the fact that we are aiming for an overall control of the type I error, which is adopted in some literature \citep[e.g.][]{maillard2019sequential}.

To obtain a test of power approaching one it appears necessary to strengthen Assumption~\ref{assump-2} as follows. For each choice of $\Delta$, $\sigma$, $\kappa$ and $\alpha$ and each $m > \Delta$, assume that 
	\begin{equation}\label{eq-asyp-snr-cond}
		\Delta \kappa^2 \sigma^{-2} \geq C_{\mathrm{SNR}} \log(m/\alpha).
	\end{equation}
	The difference between Assumption~\ref{assump-2} and the previous condition is that the right hand side is allowed to be arbitrarily large in $m$ as opposed to be fixed to $\log(\Delta/\alpha)$.  
	Next, let $\mathbb{P}_{\Delta,m}(\cdot)$ denote the probability distribution of the data under condition \eqref{eq-asyp-snr-cond}. Then it is possible to show  that 
	\[
		\lim_{m \to \infty} \mathbb{P}_{\Delta,m}(\widehat{t} < \infty) = 1.
	\]
	Of course, even under \eqref{eq-asyp-snr-cond}, we still retain the same type I error guarantees that $\mathbb{P}_\infty(\widehat{t} < \infty) < \alpha $ and $\mathbb{P}_{\Delta,m}(\widehat{t} \leq \Delta) < \alpha$.

\Cref{thm-2} could be equivalently stated using random times and, in fact, more generally, stopping times, leading to a more practical interpretation. Such formulation is better suited to accommodate real-life situations, where the experiment cannot continue indefinitely and is instead ultimately terminated based on a predefined, possibly random, stopping rule.
 Specifically, the proof of \Cref{thm-2} also implies that,  with probability at least $1- \alpha$ and uniformly over all  finite random times $T$, it holds that
\begin{itemize}
\item $\widehat{t} > T$ if $\Delta = \infty$ or if $T \leq \Delta$,
\item $\Delta < \widehat{t} \leq \Delta + C_d\frac{\sigma^2 \log(\Delta/\alpha)}{\kappa^2}$, if $T > \Delta + C_d\frac{\sigma^2 \log(\Delta/\alpha)}{\kappa^2}$.
\end{itemize}
Whenever $T$ happens to be between  $\Delta$ and $\Delta + C_d \sigma^2 \kappa^{-2}\log(\Delta/\alpha)$, then the procedure is only guaranteed to not raise a false alarm with probability at least $1 - \alpha$, at any time prior to and including $\Delta$ but may not be able to reach the correct decision that a change point has occurred since there are not enough observations between $\Delta$ and $T$ to make such a determination.

The proof of \Cref{thm-2}, as well as of most of the results in the paper, relies in a fundamental way on an auxiliary result given in Lemma~\ref{lem-alt-cumsum-prob}, which shows that, with probability at least $ 1- \alpha$, the maximal CUSUM statistic process $ t \in \{ 2,3,\ldots\} \mapsto \max_{1 \leq s < t} \widehat{D}_{s,t}$ does not  exceed the curved boundary of the form $t \mapsto 2^{3/2}\sigma \sqrt{ \log (t)/\alpha}$. The proof also shows that, for any fixed $t \geq 2$, the order of  $ \max_{1 \leq s < t} \widehat{D}_{s,t}$ is no larger than $\sigma \sqrt{ \loglog \left(t/\alpha \right)}$ with probability at least $ 1 - \alpha$. Interestingly, the iterated logarithmic scaling does not appear to apply to the maximal CUSUM process, but only to the CUSUM maximum at any fixed time $t$.  

\medskip
\Cref{thm-2} shows that \Cref{alg-online} can detect the change point with probability at least $1 - \alpha$, provided that 
	\[
		\Delta \kappa^2 \sigma^{-2} \gtrsim \log(\Delta/\alpha),
	\]
	as required in Assumption~\ref{assump-2}.	In fact, this condition is essentially necessary for our procedure. More precisely, in our next result we show that if
	\begin{equation}\label{eq:cacca.snr}
		\Delta \kappa^2 \sigma^{-2} \lesssim \log(\Delta/\alpha),
	\end{equation}
		then, with probability at least $1 - \alpha$, \Cref{alg-online}  will never terminate and therefore will be unable to reliably detect a change point. That is, under the scaling \eqref{eq:cacca.snr}, our procedure only delivers type I error control but is virtually powerless. We conjecture that the scaling in Assumption~\ref{assump-2} is required by {\it any} online procedure in order to guarantee a false alarm probability smaller than $\alpha$ and at the same time, non-trivial power (which in our problem would correspond to a probability no smaller than $ 1- \alpha$ of terminating in a finite time when there exists a change point). 

\begin{proposition}\label{prop-undetectable}
Fix an $\alpha \in (0,1)$ and let Assumptions~\ref{assump-1} and \ref{assump-2} hold.  Assume in addition that for any $\alpha \in (0, 1)$, there exists a positive constant $c_{\mathrm{SNR}} < (C_b - 2^{5/2})^2$ such that
	\begin{equation}\label{eq-low-snr-rrrr}
		\Delta \kappa^2 \sigma^{-2} \leq c_{\mathrm{SNR}} \log(\Delta/\alpha). 
	\end{equation}
Let $\widehat{t}$ be the stopping time returned by \Cref{alg-online} with inputs $\{ X_t \}_{t=1,2,\ldots}$ and $\{ b_t \}_{t=2,3,\ldots}$, where
	\[
			b_{t} = C_b\sigma\log^{1/2}(t/\alpha), \quad t \geq 2, 
		\]
		with $C_b > 2^{5/2}$. Then,
		\[
\mathbb{P}_\Delta \left( \widehat{t} < \infty \right) <   \alpha. 
		\]
\end{proposition}

\subsection{Variations of \Cref{alg-online}}\label{sec-variants}	

The framework for online change point detection we have analyzed so far  is based on controlling the false alarm probability $\alpha$, and  the magnitude of the detection delay of our procedure depends on the choice of the target probability, see \eqref{eq:thm-2.main}.  An alternative approach, considered in the literature and frequently adopted in practice, is to impose the weaker requirement that, for a target value $\gamma > 0$, $\mathbb{E}_{\infty} (\widehat{t}) 
\geq \gamma$. That is, the expected false detection time, or average run length, of an online procedure  should be no smaller than $\gamma$ when in fact there is no change point. In practical applications in which the pre-change point distribution is known or can be estimated from historical data, online procedures are often calibrated through simulations  by choosing $\gamma$ according to this criterion \citep[e.g.][]{lorden1971procedures, lai2001sequential}.

Below we show how \Cref{alg-online} can be tuned  to yield this type of false alarm control, as well as a high probability bound on the detection delay when there is in fact a change point, where the magnitude of the detection delay scales logarithmically with $\gamma$.

\begin{proposition}\label{prop-alternative-type1}
Consider the settings described in Assumption~\ref{assump-1}. Let $\gamma \geq 2$ and $\widehat{t}$ be the stopping time returned by \Cref{alg-online} with inputs $\{ X_t \}_{t=1,2,\ldots}$ and $\{ b_t \}_{t=2,3,\ldots}$, where
\begin{equation}\label{eq-bt-alter-t1}
b_{t} = 6^{1/2}\sigma\log^{1/2}\left\{2^{1/3}(\gamma+1)\right\}.
\end{equation}
If $\Delta = \infty$, then,
\[
			\mathbb{E}_{\infty}(\widehat{t}) \geq \gamma.
		\]
Under Assumptions~\ref{assump-2}, if
		\begin{equation}\label{eq-m-gamma-cond}
			\gamma \geq \Delta \quad \mbox{and} \quad  \Delta \kappa^2 \sigma^{-2} \geq C_{\mathrm{SNR}} \log(\gamma)
		\end{equation}
		where $C_{\mathrm{SNR}} > 0$ is an absolute constant, then 
		\[
			\mathbb{P}\left\{\Delta < \widehat{t} \leq \Delta + C_d\frac{\sigma^2 \log(\gamma)}{\kappa^2} \right\} \geq  1 -  \gamma^{-1},
		\]
	 where $C_d > 0$ is an absolute constant.
	  \end{proposition}

The proof of Proposition~\ref{prop-alternative-type1} is in \Cref{sec-proofs-theorems}.  Just like with \Cref{thm-2}, it is immediate to formulate a  version of Proposition~\ref{prop-alternative-type1} that is based on stopping times. For brevity, we refrain from providing details.

It is illustrative to compare the guarantees afforded by \Cref{thm-2} and Proposition~\ref{prop-alternative-type1} when a change point exists. In this case, the signal-to-noise ratio condition defined in Assumption \ref{assump-snr} is replaced by the analogous conditions shown in \eqref{eq-m-gamma-cond}, where it is crucial that $\gamma$ is an upper bound on $\Delta$. In terms of the upper bounds on the detection delay, the difference is between $\log(\gamma)$ and $\log(\Delta/\alpha)$.  When $\gamma$ is of the same order as $\Delta$, then these two detection delay upper bounds differ by a factor $\log(1/\alpha)$, which is of constant order in the fixed confidence settings where $\alpha$ is held constant.  In terms of the values of the probability bounds, the difference is between $\gamma^{-1}$ and $\alpha$.  This further suggests that, as long as 
	\[
		\gamma \asymp \alpha^{-1},\quad \gamma \geq \Delta \quad \mbox{and} \quad \gamma \asymp \Delta,
	\]
	these two strategies are equivalent in terms of controlling the detection delay. This connection has been studied before in a slightly different form, see e.g.~\cite{lai1998information}.

In Proposition~\ref{prop-alternative-type1}, if Assumption~\ref{assump-2} holds but \eqref{eq-m-gamma-cond} does not, especially if $\gamma < \Delta$, then we can only reach the conclusion that $\mathbb{E}_{\Delta}(\widehat{t}) \geq \gamma$, but no longer guarantee the control on the type I error or on the detection delay.  This is because the values of the threshold $b_t$, defined in \eqref{eq-bt-alter-t1}, is constant with respect to $t$ and only depends on $\gamma$.

In another variant of \Cref{alg-online} that is often used in practice, researchers may only wish to control the false alarm probability over any interval of a pre-specified length.  This strategy may be preferred for computational reasons,  as the computational cost of the procedure can be directly controlled by selecting an appropriate interval length, especially when training can easily be done with historical data.  Below, we provide a parallel result to \Cref{thm-2}. To that effect, we require a new definition of the CUSUM statistic in \Cref{def-cusum-multi} and  slightly modify \Cref{alg-online}, resulting in \Cref{alg-online-alter}, whose computational cost is of order $O(t)$, when proceeding to time point~$t$.  

\begin{definition}\label{def-cusum-multi}
Given data $\{X_t\}_{t = 1, 2, \ldots} \subset \mathbb{R}$ and a triplet of integers $1 \leq e < s < t$, we define the CUSUM statistic and its population counterpart as	
	\[
		\widehat{D}_{e, s, t} = \left|\left\{\frac{t-s}{(s-e)(t-e)}\right\}^{1/2}\sum_{l = e+1}^s X_l - \left\{\frac{s-e}{(t-s)(t-e)}\right\}^{1/2}\sum_{l = s+1}^t X_l\right|
	\]
	and
	\[
	D_{e, s, t} = \left|\left\{\frac{t-s}{(s-e)(t-e)}\right\}^{1/2}\sum_{l = e+1}^s f_l - \sqrt{\frac{s-e}{(t-s)(t-e)}}\sum_{l = s+1}^t f_l\right|,
	\]
	respectively. 
\end{definition}

\begin{algorithm}[!ht]
	\begin{algorithmic}
		\INPUT $\{X_u\}_{u=1, 2, \ldots} \subset \mathbb{R}$, $\gamma \geq 2$, $b_{\gamma} > 0$.
		\State $t \leftarrow 1$
		\While{$\mathrm{FLAG} = 0$}  
			\State{$t \leftarrow t + 1$;}
			\State{$\mathrm{FLAG} = 1 - \prod_{s = \max\{t-\gamma+1, \, 1\}}^{t-1} 1 \left\{\widehat{D}_{\max\{t-\gamma, \, 0\}, s, t} \leq b_{\gamma}\right\}$;}
		\EndWhile
		\OUTPUT $t$.
		\caption{Online change point detection 2.} \label{alg-online-alter}
	\end{algorithmic}
\end{algorithm} 

\begin{proposition}\label{prop-alter-2}
Consider the settings described in Assumption~\ref{assump-1}. Let $\alpha \in (0, 1)$ and $\widehat{t}$ be the stopping time returned by \Cref{alg-online-alter} with inputs $\{ X_t \}_{t=1,2,\ldots}$, $\gamma \geq 2$ and $b_{\gamma} = 2^{1/2}\sigma\log^{1/2} \left(2\gamma^2/\alpha\right)$.  Then, when $\Delta = \infty$,
\[
			\sup_{v \geq 1} \mathbb{P}_{\infty}\left(v < \widehat{t} \leq v + \gamma \right) < \alpha.
		\]
If Assumptions~\ref{assump-2} and \ref{assump-snr} hold and 
		\begin{equation}\label{eq-alter-gamma-cond}
			\gamma > C_{\gamma}\sigma^2 \log(2\gamma^2/\alpha) \kappa^{-2} ,		
		\end{equation}
		with an absolute constant $C_{\gamma} > 0$ then with an absolute constant $C_d > 0$,
\[
			\mathbb{P}\left\{ (\widehat{t} - \Delta)_+\geq C_d\sigma^2 \log(2\gamma^2/\alpha) \kappa^{-2} \right\} < \alpha.
		\]
\end{proposition}

The proof of Proposition~\ref{prop-alter-2} can be found in \Cref{sec-proofs-theorems}.  The strategy used in proving Proposition~\ref{prop-alter-2} can be seen as a mixture of the strategies used in \Cref{thm-2} and Proposition~\ref{prop-alternative-type1}: instead of controlling the  type I error over the whole time course, it is enough to only control the type I error over intervals of length $\gamma$. Comparing with Proposition~\ref{prop-alternative-type1}, the advantage here is that, if $\gamma < \Delta$, we will be able to provide a high-probability bound on the detection delay.  As for \Cref{thm-2}, the advantage here is the same as that in Proposition~\ref{prop-alternative-type1}, that it might be handier in tuning parameter selection, while the price it pays is here that,  when $\gamma < \Delta$, then there is no guarantee that $\widehat{t} > \Delta$, i.e.~preventing false alarms. 

\subsection{Practical and computational issues}\label{sec-practical}

In this section we are to discuss two practical issues for the CUSUM-statistics-based online procedures discussed above. 

The first issue is the choice of tuning parameters.  In \Cref{alg-online}, we need a sequence of tuning parameters $\{b_t\}$, whose theoretically-justified value is $2^{3/2}\sigma\log^{1/2}(t/\alpha)$, see  \eqref{eq-b-2}. 
The quantity $t$ and $\alpha$ can be determined by the user, but the sub-Gaussian parameter $\sigma$ remains unknown.   In some situations, the practitioner  has access to independent copies of data generated from a model with no change points.  Then, one  may wish to set a limit of time, say $T$, and estimate the empirical type I errors over the time course $\{1, \ldots, T\}$, in order to tune the thresholds.    In this sense, the first variant we mentioned in \Cref{sec-variants} is handier.  The tuning parameter can be chosen by setting the average run length equal to a pre-specified $\gamma$.

The second  practical issue is computational complexity.  The CUSUM statistic given in \Cref{def-cusum} can be rewritten as
	\[
		\widehat{D}_{s, t} = \left\{\frac{t}{s(t-s)}\right\}^{1/2}\left|\frac{s}{t} \sum_{i = 1}^t X_i - \sum_{i = 1}^s X_i\right|.
	\]
	Using this equivalent expression, for each time $t$, one can store all partial sums $\{\sum_{i = 1}^s X_i\}_{s = 1}^t$ and the computational cost of \Cref{alg-online} is therefore of order $O(t)$ but the storage is also of order $O(t)$.  As an alternative, one may elect to recalculate the CUSUM statistic every time, in which case there is no requirement on storage but the computational cost increases to be of order $O(t^2)$.
	
	To reduce the computational burden, a very simple strategy is to avoid calculating the maximal values of the CUSUM statistics at all integer pairs $1 \leq s < t$ and instead consider only pairs that are at a certain distance in time, say $h>0$, a pre-specified parameter picked by the user. The window width $h$ can be regarded as the user's maximal  tolerance on accuracy. More precisely, at each time  $t \geq 2$, instead of maximizing the values of the CUSUM statistics over all integers $s \in [1, t)$, one could just calculate $\widehat{D}_{t - h, t}$, for 
	\[
		h \asymp \frac{\sigma^2}{\kappa^2 \log(\Delta)} \quad \mbox{and} \quad t > h.
	\]
		Then, in \Cref{alg-online}, it would only be enough to check if $\widehat{D}_{t - h, t}$ exceeds an appropriate threshold.  The computational complexity of this alternative is of order $O(t)$ and the storage cost is of order $O(1)$.  The caveat of this alternative is that one needs to carefully tune $h$, which is required to essentially have the same order of the magnitude as the detection delay.  To address this issue, we propose  to compute $\widehat{D}_{t - h, t}$ only over a geometrically increasing sequence of values for $h$, a tuning strategy that will only incur in an additional computational cost of order $O\{\log(t)\}$ when proceeding to time point  $t$. The detailed procedure is given in \Cref{alg-online-net-3}. As it turns out, this strategy is not only computationally convenient, but it yields the same nearly optimal theoretical guarantees of \Cref{thm-2}, as shown next.

\begin{algorithm}[!ht]
	\begin{algorithmic}
		\INPUT $\{X_u\}_{u=1, 2, \ldots} \subset \mathbb{R}$, $\{b_u, u = 1, 2, \ldots,\} \subset \mathbb{R}$.
		\State $t \leftarrow 1$
		\While{$\mathrm{FLAG} = 0$}  
			\State{$t \leftarrow t + 1$;}
			\State{$J \leftarrow \lfloor \log(t)/\log(2)\rfloor$;}
			\State{$j \leftarrow 0$;}
			\While{$j < J$ and $\mathrm{FLAG} = 0$}
				\State{$j \leftarrow j + 1$;}
				\State{$s_j \leftarrow t-2^{j-1}$;}
				\State{$\mathrm{FLAG} = 1\left\{\widehat{D}_{s_j, t} > b_{t}\right\}$;}
			\EndWhile
		\EndWhile
		\OUTPUT $t$.
		\caption{Online change point detection - varying $h$.} \label{alg-online-net-3}
	\end{algorithmic}
\end{algorithm}

\begin{corollary} \label{prop-new}
Consider the settings described in Assumption~\ref{assump-1}. Let $\alpha \in (0, 1)$ and $\widehat{t}$ be the stopping time returned by \Cref{alg-online-net-3} with inputs $\{ X_t \}_{t=1,2,\ldots}$ and $\{ b_t \}_{t=2,3,\ldots}$, where
	\[
			b_{t} = 2^{3/2}\sigma\log^{1/2}(t/\alpha), \quad t \geq 2.
		\]
		 If $\Delta = \infty$, then
		\[
		\mathbb{P}_{\infty} \left\{ \widehat{t} < \infty \right\} < \alpha.
		\]
		Under Assumption \ref{assump-2}, we have
		\[
		\mathbb{P}_{\Delta} \left\{ \widehat{t} \leq \Delta \right\} < \alpha,
		\]
		for any $\Delta \geq 1$. 
		If Assumptions \ref{assump-2} and \ref{assump-snr} both hold, then 
\[
		\mathbb{P}\left\{\Delta <  \widehat{t} \leq  \Delta + C_d\frac{\sigma^2 \log(\Delta/\alpha)}{\kappa^2}\right\} \geq 1 - \alpha,
	\]
	for all $\Delta \geq 1$, where $C_d > 0$ is an absolute constant satisfying $C_d < C_{\mathrm{SNR}}$.
\end{corollary}

\section{Multiple change points}\label{sec-multi}

It is relatively straightforward to extend our methodology and analysis to deal with multiple change points. In this setting, the analyst still collects data sequentially and, with each new data point, makes a decision as to whether there exists evidence supporting the presence of a change point in the near past. But, unlike the single change point case considered in the previous sections, the procedure is restarted each time a new change point is declared, until the experiment is terminated.  In order to accommodate for the presence of multiple change points, we need a refined setup.

\begin{assumption}\label{assump-multi}
Assume that there exists a collection of change points $ 1 = \eta_0  <  \eta_1  <  \eta_2 <  \ldots$, such that
	\[
		f_{\eta_{k-1}} = \cdots = f_{\eta_k - 1} = \mu_k, \quad k = 1, 2, \ldots,
	\]
	where $\mu_k \neq \mu_{k+1}$ for each $k=1,2,\ldots$.  Let
	\[
		\Delta = \inf_{k = 1, 2, \ldots} (\eta_k - \eta_{k-1})
	\]	
	and, for each $k=1,2,\ldots$,
	\[
		\kappa_k  = |\mu_k - \mu_{k+1}|.  
	\]
	Finally, set $\kappa = \inf_{k = 1, 2, \ldots} \kappa_k$, assumed to be strictly positive.
\end{assumption}

In Assumption~\ref{assump-multi}, the minimal spacing between two consecutive change points $\Delta$ and the minimal jump size $\kappa$ are natural generalizations of the analogous quantities defined in Assumption~\ref{assump-2}  and indeed coincide with them when there is only one change point.

\Cref{alg-online-multi} is an immediate generalization of \Cref{alg-online} and amounts to repeatedly applying \Cref{alg-online} as soon as a new change point is found. Below we will demonstrate that, under the signal-to-noise ratio condition described in Assumption~\ref{assump-snr-2}, given a pre-specified confidence level $\alpha \in (0,1)$, the input parameters $\{b_u, u = 2, 3, \ldots\}$ of the algorithm can be chosen so that, with probability at least $1- \alpha$, 
\begin{itemize}
\item the procedure will not declare any change point if in fact there is none;
\item if there are change points, then, for any stopping time $T$ with respect to the natural filtration induced by the data, the procedure will not return any change point if $T \leq \eta_1$ and will identify the sequence $\{\eta_1,\ldots,\eta_K\} =  \{ 1, \ldots, T\} \cap \{ \eta_k \}_{k=1,2,\ldots}$ of true change points prior to~$T$. Furthermore, in the former case, the detection delays will be of order $O\left\{\sigma^2\kappa_j^{-2} \log (\Delta/\alpha) \right\}$, $j=1,\ldots K$. It is worth noting that the number $K$ will in general be random if the stopping time $T$ is not deterministic.
\end{itemize}

\begin{algorithm}[!ht]
	\begin{algorithmic}
		\INPUT $\{X_u\}_{u=1, 2, \ldots} \subset \mathbb{R}$, $\{b_u, u = 2, 3, \ldots\} \subset \mathbb{R}$, $\mathcal{C} = \emptyset$.
		\State $\mathrm{FLAG} = 0$
		\State $e \leftarrow 0$
		\State $t \leftarrow 1$
		\While{there is a new data point}  
			\State{$t \leftarrow t + 1$;}
			\State{$\mathrm{FLAG} = 1 - \prod_{s = 1}^{t-1} \mathbbm{1} \left\{\widehat{D}_{e, s, t} \leq b_t\right\}$;}
			\If{$\mathrm{FLAG} = 1$}
				\State $\mathcal{C} \leftarrow \mathcal{C} \cup \{t\}$;
				\State $\mathrm{FLAG} \leftarrow 0$;
				\State $e \leftarrow t$
			\EndIf
		\EndWhile
		\OUTPUT $\mathcal{C}$.
		\caption{Online change point detection - multiple change points.} \label{alg-online-multi}
	\end{algorithmic}
\end{algorithm}

The signal-to-noise condition that we required for the case of multiple change points is as follows.

\begin{assumption}\label{assump-snr-2}
Assume that for any $\alpha \in (0, 1)$, there exists a sufficiently large absolute constant $C_{\mathrm{SNR}} > 0$ such that
	\[
		\Delta \kappa^2 \sigma^{-2} \geq C_{\mathrm{SNR}} \log(\Delta/\alpha).
	\]	
\end{assumption}

Assumptions~\ref{assump-snr} and \ref{assump-snr-2} are identical in appearance, though, of course, the meaning of the quantities $\Delta$ and $\kappa$ are different depending on the settings.
In the main result of this section we generalize the findings in \eqref{eq:thm-2.main} to the settings involving multiple change points, deriving essentially the same false alarm and detection delay bounds as in the only one change point case. In order to allow for a more transparent statement, we have expressed our results in terms of a uniform choice of stopping times for terminating the experiment.

\begin{theorem}\label{thm-multi}
Consider the settings described in Assumption~\ref{assump-1}. Let $\alpha \in (0, 1)$ and $\mathcal{C}$ be the output of \Cref{alg-online-multi} with inputs $\{X_t\}_{t = 1, 2, \ldots}$ and $\{b_u, u = 2, 3, \ldots\}$, where
\begin{equation}\label{eq:C4}
			b_u = 4\sigma\log^{1/2}(t/\alpha), \quad \forall u = 2,3,\ldots.
		\end{equation}

The following holds with probability at least $ 1- \alpha$, uniformly over all choices of finite stopping times $T$ with respect to the natural filtration generated by  $\{X_t\}_{t = 1, 2, \ldots}$.

\begin{itemize}
\item If the random variables $\{X_t\}_{t = 1, 2, \ldots}$ have constant means, then $\mathcal{C} = \emptyset.$

\item Under Assumptions \ref{assump-multi} and \ref{assump-snr-2}, let $K = |\{1, \ldots, T\} \cap \{\eta_k, k = 1, 2, \ldots\}|$. 
If $K=0$, then 
\[
\mathcal{C} \cap \{1,\ldots,T\} = \emptyset.
\] 
If $K \geq 1 $, there exists an absolute constant $C_d > 0$ satisfying $C_d < C_{\mathrm{SNR}}$ such that
\[
\mathcal{C} = \begin{cases}
	\{ \hat{t}_1, \ldots, \hat{t}_{K-1}\} \text{ or } \{ \hat{t}_1, \ldots, \hat{t}_{K}\} & \text{if } 0 \leq T - \eta_K < C _d\frac{\sigma^2 \log(\Delta/\alpha)}{\kappa_k^2},\\ 	
	\{ \hat{t}_1, \ldots, \hat{t}_{K}\}  & \text{otherwise,}
 \end{cases}
\]
where we set $\{ \hat{t}_1, \ldots, \hat{t}_{K-1}\} = \emptyset$ when $K=1$. 
 Furthermore,
\[
 0 \leq \hat{t}_k - \eta_k \leq C_d\frac{\sigma^2 \log(\Delta/\alpha)}{\kappa_k^2}, \quad \forall k = 1,\ldots, K,
\]
where, if $\mathcal{C} = \{ \hat{t}_1, \ldots, \hat{t}_{K-1}\}$, we define $\hat{t}_K = T$.
	\end{itemize}
\end{theorem}

The proof of \Cref{thm-multi} is in \Cref{sec-proofs-theorems}.  The proof of \Cref{thm-multi} repeatedly uses the proof and results of \Cref{thm-2}, therefore the proof of \Cref{thm-multi} holds for any fixed time point $T$, including random finite stopping times $T$.  In fact, Lemma 3 in \cite{howard2018uniform} also shown such equivalence in the context of sequential testing.

One of the key steps in the proof is to show that, for any $k \geq 1$, 
	\begin{equation}\label{eq:distance}
		0 \leq  \widehat{t}_k - \eta_k \leq  C_d\frac{\sigma^2 \log(\Delta/\alpha)}{\kappa_k^2} < \Delta/4.
	\end{equation}
	Once a change point is declared, the procedure is restarted afresh using the latest change point estimate as the new initial time.  Due to the last display, we are guaranteed that the new starting point of the procedure, namely $\widehat{t}_k$, is only slightly delayed, by an amount no larger than $\Delta/4$ from the ideal starting point $\eta_k$.  The constant of $1/4$ is not special and could be replaced by any constant smaller than $1/2$.  Equivalently, the next change point will be sufficiently detected starting from   $\widehat{t}_k$; precisely, we have that
	\[
		\eta_{k+1} - \widehat{t}_k = \eta_{k+1} - \eta_k - (\widehat{t}_k - \eta_k) > 3\Delta/4.
	\]
	As demonstrated in our proof, this delay is sufficient, under the signal-to-noise condition in Assumption~\ref{assump-snr-2}, to ensure that the procedure is able to detect the next change point $\eta_{k+1}$ and to localize it with a detection delay no larger than $C_d\sigma^2 \kappa_{k+1}^{-2}\log(\Delta/\alpha)$.

	It is important to notice that the quantity $K$ in the theorem statement is itself a random variable. Furthermore,   the theorem only guarantees that the procedure will be able to identify, with probability at least $ 1- \alpha$ all the $K$ true change point in the interval $\{1,\ldots,T\}$ as long as $T - \eta_K \geq C_d \sigma^2 \kappa_k^{-2}\log(\Delta/\alpha)$. If this condition is not satisfied, i.e.~when the difference between the $k$th time point and the stopping time is too small in relation to the assumed signal-to-noise ratio Assumption~\ref{assump-snr-2}, there is not guarantee that the procedure will detect the last change point $\eta_K$. This is not surprising: the data collected  in such a short amount of time may not in general contain enough information to support that decision.

\section{Optimality}\label{sec-lowerbounds}

\subsection{A lower bound on the detection delay}
In this section we show that the detection delay bounds derived in \Cref{thm-2} are essentially sharp. Towards this end, we adapt to our settings existing techniques for deriving  minimax optimality from the literature on change point analysis, which however only apply asymptotically and under parametric assumptions. In particular, Theorem 1 of \cite{lorden1971procedures} shows that,  assuming independent Gaussian instead of sub-Gaussian data,  
\[
\inf_{\widehat{t}} \sup_{\Delta \geq 1} \esssup \mathbb{E}_{\Delta} \left\{(\widehat{t}-\Delta)^+ \mid X_1, \ldots, X_{\Delta}\right\} \sim \frac{\sigma^2\log(\gamma)}{\kappa^2}, \quad \text{as } \gamma \to \infty.
	\]
where the infimum is taken over all change point estimators $\widehat{t}$ such that $\mathbb{E}_{\infty}(\widehat{t}) \geq \gamma$.
	The last display bears clear similarities with the high-probability delay bounds in \Cref{sec-upperbounds}  and,  in particular, with the guarantees obtained in Proposition~\ref{prop-alternative-type1}.

In our next non-asymptotic result, we formally show that the upper bound on the detection delay of \Cref{thm-2} is in fact in agreement, aside from a $\log(\Delta)$ factor, with a minimax lower bound on the expected detection delay. The proof adapts arguments used  in Theorem 2 of \cite{lai1998information}.  

\begin{proposition}\label{prop-lb}
	Assume that $\{X_i\}_{i = 1, 2, \ldots}$ is a sequence of independent Gaussian random variables satisfying $\mathbb{E}(X_i) = f_i$, $\mathrm{Var}(X_i) = \sigma^2$ and Assumption~\ref{assump-2}.  Denote the joint distribution of $\{X_i\}_{i = 1, 2, \ldots}$ as $P_{\kappa, \sigma, \Delta}$.  For $\alpha \in (0,1)$, consider the class of change point estimators
	\begin{align*}
		& \mathcal{D}(\alpha)  = \Big\{T:\, T \mbox{ is a stopping time with respect to the natural filtration} \\
		& \hspace{4cm}\mbox{and satisfies } \mathbb{P}_{\infty}(T < \infty) \leq \alpha \Big\}.
	\end{align*}
 Then for all $\alpha$, $\kappa$ and $\sigma$ such that $\sigma^2(4\kappa^2)^{-1} \log\left(1/\alpha\right) \geq 3$,
\begin{equation}\label{eq-delta-const}
		\alpha + 2\alpha^{1/4} < 1/2 \quad \mbox{and} \quad \alpha^{5/4}\log(1/\alpha) \leq 2\kappa^2\sigma^{-2},
	\end{equation}
 it holds that, for any change point time $\Delta$,
	\begin{equation}\label{eq-prop-lb-result}
		\inf_{\widehat{t} \in \mathcal{D}(\alpha)} \sup_{P_{\kappa, \sigma, \Delta}}\mathbb{E}_P\left\{(\widehat{t} - \Delta)_+\right\} \geq \frac{\sigma^2}{4\kappa^2} \log\left(\frac{1}{\alpha}\right).
	\end{equation}
\end{proposition}

\begin{remark}
The constants in \eqref{eq-delta-const} have not been optimized and are likely to be improvable.
\end{remark}

Although Proposition~\ref{prop-lb} provides a lower bound for the expected detection delay, inspection of its proof in fact reveals that,  for all $\alpha$ small enough and any stopping time $T \in \mathcal{D}(\alpha)$,
\[
\mathbb{P}_{\kappa, \sigma, \Delta} \left\{T \geq \Delta + \frac{\sigma^2}{4\kappa^2} \log\left(\frac{1}{\alpha}\right) \right\} \geq  1/2.
\]
The bound in the last display differs from the probabilistic upper bound on the detection delay in \Cref{thm-2} by a $\log(\Delta)$ term. Thus, aside possibly from a quantity of that order, the detection delay suffered by our main procedure in \Cref{alg-online} achieves the  minimax rate.

  In the context of multiple change points, the lower bound in Proposition~\ref{prop-lb} further yields that the signal-to-noise ratio condition in Assumption~\ref{assump-snr-2} is almost necessary to guarantee a consistent change point localization. As remarked after the statement of \Cref{thm-multi}, with probability at least $1 - \alpha$, under this condition, the $k$th estimated change point is at a distance of no more than $\Delta/4$ from the $k$th largest true change point, see \eqref{eq:distance}. In contrast, whenever Assumption~\ref{assump-snr-2} is violated in such a way that 
	\begin{equation}\label{eq-low-snr}
		\kappa^2 \Delta/\sigma^2 < \log(1/\alpha)/4,
	\end{equation}
	then \eqref{eq-prop-lb-result} implies that 
	\[
		\frac{\inf_{\widehat{t} \in \mathcal{D}(\alpha)} \sup_{P_{\kappa, \sigma, \Delta}} \mathbb{E}_P\left\{(\widehat{t} - \Delta)_+\right\}}{\Delta} > 1.
	\]
	Thus the expected detection delay is larger than $\Delta$ and, therefore, the true change point cannot be consistently estimated. 

	Overall, we may conclude that  the online multiple change point detection problem is impossible under the low signal-to-noise ratio condition $\kappa^2 \Delta/\sigma^2 \lesssim \log(1/\alpha)$. In contrast, a slightly larger signal-to-noise ratio $\kappa^2 \Delta/\sigma^2 \gtrsim \log(\Delta/\alpha)$ is sufficient to guarantee, as shown in \Cref{thm-multi}, that the change points can be accurately estimated with probability no smaller than the prescribed nominal level of $1- \alpha$. The difference between the two signal-to-noise conditions is a factor of order $\log(\Delta)$, thus showing that the procedure described in \Cref{alg-online-multi} succeeds even in the most difficult settings where the signal-to-noise ratio approaches the infeasibility threshold. A similar phase transition phenomenon occurs for the offline mean change localization problem, see \cite{wang2018univariate}.  In fact, the phase transition boundary is nearly identical in both the online and offline settings. 
	
\subsection{Connections with offline change point detection problems}\label{sec-offline}

A closely related area is the offline change point detection, where one has data $\{X_i\}_{i = 1}^T$ and seeks change point estimators $\{\widehat{\eta}_k\} \subset \{1, \ldots, T\}$.  The online and offline change point analysis shares many similarities.  The offline change point results we list below can be found in \cite{wang2018univariate}.

The signal-to-noise ratio and the phase transition.  Let $\Delta$ be the minimal spacing between two consecutive change points in the offline setting.  We remark that in both online and offline problems, the signal-to-noise ratio is of the same form $\kappa^2 \Delta /\sigma^2$.  In both problems, the parameter spaces can be partitioned into feasibility and infeasibility regimes by this signal-to-noise ratio.  In \cite{wang2018univariate}, it is shown that in the univariate offline change point detection problem, both the lower and upper bounds on the signal-to-noise ratio are of order $\log(T)$.  This sheds some light that the logarithmic factor gap between the lower bound we established in Proposition~\ref{prop-lb} and \eqref{eq-low-snr}, and the upper bound we assumed in Assumption~\ref{assump-snr}, is due to a loose lower bound but not the upper bound.

The estimation errors in both these two problems have a minimax lower bound $\sigma^2/\kappa^2$, and the upper bounds we achieve are both nearly optimal, off possibly by a logarithmic factor.

When deriving the estimation error, since one has collected all the data in advance in the offline setting, the signal-to-noise ratio is lower bounded by $\log(T)$ and as a result, the estimation error only depends on the model parameters.  This is not the case in the online setting, where the total number of data points examined is also a random variable.  In this case, additional information is needed.  In \Cref{thm-2}, we choose to control the upper bound of the type I error $\alpha$.  As a result, the estimation error, i.e.~the detection delay, is a function also of $\alpha$.

\section{Discussion}\label{sec-discussions}

We have investigated various aspects of online mean change point estimation with sub-Gaussian, independent variates. In our analysis, we have adopted finite sample settings and allowed for the model parameters to vary so as to express all the sources of statistical hardness in the problem. We have presented simple but effective algorithms for online change point localization based on classical CUSUM process that are provably nearly optimal. We have also identified combinations of model parameters for which the change point estimation task is impossible and showed that our procedure succeeds over nearly all the other combinations, including those for which the problem is most difficulty and barely solvable. Interestingly, the online estimation rates and phase transition boundaries we determine are nearly identical to those holding in the offline settings, see \cite{wang2018univariate}. The framework we established in this paper can be used to study different high-dimensional problems.  This will be left as future work. 

In order to control the fluctuations of the CUSUM statistics, we have adopted two different approaches.  Theorems~\ref{thm-2} and \ref{thm-multi} rely on peeling arguments, as detailed in Lemmas~\ref{lem-alt-cumsum-prob} and \ref{lem-concen-multi}.  In contrast, the results  in \Cref{sec-variants} rely on simpler union bound arguments. Interestingly, even using peeling, we were only able to establish that the CUSUM statistic process
\[
		t \mapsto \sup_{t \geq 2}\max_{1 \leq s < t}\left|\left\{\frac{s}{t(t-s)}\right\}^{1/2} \sum_{l = s+1}^t (X_l - f_l)\right|,
	\]
has fluctuations of order $\sqrt{ \{\log (t) + \log(\alpha) \}/t}$. One may have expected to see a better dependence on $t$, such as in $\sqrt{\loglog(t)}$. However, the above CUSUM process does not have a clear martingale structure. As a result, peeling techniques only yield crossing boundaries that scale like $\sqrt{\log(t)}$, albeit with better constants than those of the same order stemming from union bound arguments. Since we have not made attempts to optimize constants, both approaches -- peeling and union bounds --  deliver the same localzation rates in our settings.

\appendix

\section{Concentration inequalities}\label{sec-concentration-inequalities}

\begin{lemma}\label{lem-alt-cumsum-prob}
For any $\alpha > 0$, it holds that
	\begin{align*}
		& \mathbb{P}\Bigg\{ \exists s, t \in \mathbb{N},\, t > 1, \, s \in [1, t): \, \left|\left(\frac{t-s}{ts}\right)^{1/2} \sum_{l = 1}^s (X_l - f_l) - \left\{\frac{s}{t(t-s)}\right\}^{1/2} \sum_{l = s+1}^t (X_l - f_l)\right|\\
		& \hspace{2cm} > 2^{3/2}\sigma \log^{1/2}(t/\alpha)\Bigg\} \leq \alpha.
	\end{align*}
\end{lemma}

\begin{proof}
It holds that for any sequence $\{\varepsilon_t > 0\}$,
	\begin{align*}
		& \mathbb{P}\Bigg\{\exists s, t \in N,\, t > 1, \, s \in [1, t): \, \Bigg|\left(\frac{t-s}{ts}\right)^{1/2} \sum_{l = 1}^s (X_l - f_l) \\
		& \hspace{2cm} - \left\{\frac{s}{t(t-s)}\right\}^{1/2} \sum_{l = s+1}^t (X_l - f_l)\Bigg| \geq \varepsilon_t\Bigg\}	 \\
		\leq & \sum_{j = 1}^{\infty} \mathbb{P}\left\{\max_{2^j \leq t < 2^{j+1}} \max_{1 \leq s < t} \left|\left(\frac{t-s}{ts}\right)^{1/2} \sum_{l = 1}^s (X_l - f_l) - \left\{\frac{s}{t(t-s)}\right\}^{1/2} \sum_{l = s+1}^t (X_l - f_l)\right| \geq \varepsilon_t\right\} \\
		\leq & \sum_{j = 1}^{\infty} 2^j \max_{2^j \leq t < 2^{j+1}} \mathbb{P}\left\{\max_{1 \leq s < t} \left|\left(\frac{t-s}{ts}\right)^{1/2} \sum_{l = 1}^s (X_l - f_l) - \left\{\frac{s}{t(t-s)}\right\}^{1/2} \sum_{l = s+1}^t (X_l - f_l)\right| \geq \varepsilon_t\right\} \\
		\leq & \sum_{j = 1}^{\infty} 2^j \max_{2^j \leq t < 2^{j+1}} t \mathbb{P} \left\{\left|W\right| \geq \varepsilon_t\right\} \leq \sum_{j = 1}^{\infty} 2^{2j+1} \mathbb{P} \left\{\left|W\right| \geq \varepsilon'_j\right\},
	\end{align*}
	where $W$ is a mean zero sub-Gaussian random variable with $\|W\|_{\psi_2} \leq \sigma$ and $\varepsilon_t = \varepsilon'_j$, for any $t \in \{2^j, \ldots, 2^{j+1} - 1\}$, $j = 1, 2, \ldots$.  For any $t = 1, 2, \ldots$, let
	\[
		\varepsilon_t = \sqrt{2}\sigma \left[2\log(t) + \loglog(t) + \log\left\{\log(t) + \log(2)\right\}- 2\loglog(2) - \log(\alpha)\right]^{1/2}.
	\]
	Due to the sub-Gaussianity, we have that for any $\zeta > 0$, $\mathbb{P} \left\{\left|W\right| \geq \zeta \right\} < 2\exp(-2^{-1}\zeta^2/\sigma^2)$ \citep[e.g.~(2.9) in][]{wainwright2019high}. 
	\begin{align*}
		& \mathbb{P}\Bigg\{\exists s, t \in N,\, t > 1, \, s \in [1, t): \, \Bigg|\left(\frac{t-s}{ts}\right)^{1/2} \sum_{l = 1}^s (X_l - f_l) \\
		& \hspace{2cm} - \left\{\frac{s}{t(t-s)}\right\}^{1/2} \sum_{l = s+1}^t (X_l - f_l)\Bigg| \geq \varepsilon_t\Bigg\}	\\
		\leq & \sum_{j = 1}^{\infty} \max_{2^j \leq t \leq 2^{j+1}}\exp [(2j+2)\log(2) - 2\log(t) - \loglog(t) - \log\{\log(t) + \log(2)\} \\
		& \hspace{2cm}+ 2\loglog(2) + \log(\alpha) ] \\
		\leq & \sum_{j = 1}^{\infty} \exp \Bigg\{(2j+2)\log(2) - 2j\log(2) - 2\log(j) - \log\{(j+1)\log(2)\} \\
		& \hspace{2cm} + 2\loglog(2) + \log(\alpha)\Bigg\}  \\
		\leq & \alpha \sum_{j = 1}^{\infty} \frac{1}{j(j+1)} \leq \alpha.
	\end{align*}
	For simplicity, we let 
	\[
		\varepsilon_t =  2^{3/2}\sigma \log^{1/2}(t/\alpha),
	\]
	which satisfies that for any $t \geq 2$ and $\alpha \in (0, 1)$,
	\begin{align} 
		& 2^{3/2}\log^{1/2}(t/\alpha) \geq \sqrt{2}[2\log(t) + \loglog(t) + \log\left\{\log(t) + \log(2)\right\} \nonumber \\
		& \hspace{2cm} - 2\loglog(2) - \log(\alpha)]^{1/2}.\label{eq-C-1-cond}
	\end{align}
	We therefore completes the proof.
\end{proof}

\begin{lemma}\label{lem-concen-multi}
For any $\alpha > 0$, it holds that
	\begin{align*}
	& \mathbb{P}\Bigg\{\exists e, s, t \in \mathbb{N},\, e \geq 0, \, s > e, t > s:\, \Bigg|\left\{\frac{t-s}{(s-e)(t-e)}\right\}^{1/2} \sum_{l =e+1}^s (X_l - f_l) \\
	& \hspace{0.5cm}  - \left\{\frac{s-e}{(t-s)(t-e)}\right\}^{1/2} \sum_{l = s+1}^t (X_l - f_l)\Bigg| > 4\sigma\log^{1/2}(t/\alpha) \Bigg\} \leq \alpha.
	\end{align*}	
\end{lemma}

\begin{proof}
For any integer triplet $(e, s, t)$, $0 \leq e < s < t$, any sequence $\{\varepsilon_{e, s, t} > 0\}$, let 
	\begin{align*}
		\mathcal{G}(e, s, t) & = \Bigg\{\Bigg|\left\{\frac{t-s}{(s-e)(t-e)}\right\}^{1/2} \sum_{l =e+1}^s (X_l - f_l) \\
		& \hspace{2cm} - \left\{\frac{s-e}{(t-s)(t-e)}\right\}^{1/2} \sum_{l = s+1}^t (X_l - f_l)\Bigg|  > \varepsilon_{e, s, t}\Bigg\} \\
		& = \left\{G(e, s, t) > \varepsilon_{e, s, t}\right\}.
	\end{align*}

We have that
\begin{align*}
	& \mathbb{P}\left\{\exists e, s, t \in N,\, e \geq 0, \, s > e, t > s:\, \mathcal{G}(e, s, t) \right\} \\
	\leq & \mathbb{P}\left\{\exists s, t \in N,\, e = 0, \, s > e, t > s:\, \mathcal{G}(0, s, t) \right\} \\
	& \hspace{2cm} + \mathbb{P}\left\{\exists s, t \in N,\, e \geq 1, \, s > e, t > s:\, \mathcal{G}(e, s, t) \right\}\\
	= & (I) + (II).
\end{align*}

As for $(I)$, it follows from Lemma~\ref{lem-alt-cumsum-prob} that, with $\varepsilon_{0, s, t}$ defined as 
	\begin{align*}
		& \varepsilon_{0, s, t} = \varepsilon_t^{(1)} = \sqrt{2}\sigma [2\log(t) + \loglog(t) + \log\{\log(t) + \log(2)\} \\
		& \hspace{2cm} - 2\loglog(2) + \log(2) - \log(\alpha)]^{1/2},
	\end{align*}
	we have that 
	\[ 
		(I) \leq \alpha/2.
	\]
	
As for $(II)$, we have that 
	\begin{align*}
		(II) & \leq \sum_{j = 1}^{\infty}\mathbb{P}\left\{\max_{2^j \leq t < 2^{j+1}} \max_{1 \leq e < t-1} \max_{e < s < t} G(e, s, t) > \varepsilon_{e, s, t}\right\} \\
		& \leq \sum_{j = 1}^{\infty} \sum_{m = 0}^{j-1} \mathbb{P}\left\{\max_{2^j \leq t < 2^{j+1}} \max_{2^m \leq e < 2^{m+1}} \max_{e < s < t} G(e, s, t) > \varepsilon_{e, s, t}\right\} \\
		& \hspace{2cm} + \sum_{j = 1}^{\infty} \mathbb{P}\left\{\max_{2^j \leq t < 2^{j+1}} \max_{2^j \leq e < t} \max_{e < s < t} G(e, s, t) > \varepsilon_{e, s, t}\right\} \\
		& < \sum_{j = 1}^{\infty} \left(\sum_{m = 0}^{j-1} 2^j 2^m 2^{j+1} + 2^j 2^j 2^j\right) \max_{\substack{e < s < t \\ 2^j \leq t < 2^{j+1}}} \mathbb{P}\{\mathcal{G}(e, s, t)\} \\
		& \leq \sum_{j = 1}^{\infty} 2^{3j + 2} \max_{\substack{e < s < t \\ 2^j \leq t < 2^{j+1}}} \mathbb{P}\{\mathcal{G}(e, s, t)\}.
	\end{align*}
For any $e > 0$, let
	\begin{align*}
		& \varepsilon_{e, s, t} = \varepsilon_t^{(2)} = \sqrt{2}\sigma [3\log(t) + \loglog(t) + \log\{\log(t) + \log(2)\} \\
		& \hspace{2cm} + 3\log(2) - 2\loglog(2) - \log(\alpha)]^{1/2}.
	\end{align*}
	Due to the sub-Gaussianity, we have that
	\[
		(II) \leq \alpha/2.
	\]

For simplicity, we let 
	\[
		\varepsilon_t = 4\sigma\log^{1/2}(t/\alpha),
	\]	
	such that
	\begin{align} 
		& 4\log^{1/2}(t/\alpha) \geq 2^{1/2}[3\log(t) + \loglog(t) + \log\{\log(t) + \log(2)\} \nonumber \\
		& \hspace{2cm} + 3\log(2) - 2\loglog(2) - \log(\alpha)]^{1/2}. \label{eq-C-4-conddd}
	\end{align}
	We therefore completes the proof.
\end{proof}

\section{Proofs of the main results}\label{sec-proofs-theorems}

\begin{proof}(of \Cref{thm-2})
\noindent Step 1. Define the event 
	\begin{align*}
	& \mathcal{B} = \Bigg\{\forall s, t \in N,\, t > 1, \, s \in [1, t):\, \Bigg|\left(\frac{t-s}{ts}\right)^{1/2} \sum_{l = 1}^s (X_l - f_l) \\
	& \hspace{2cm} - \left\{\frac{s}{t(t-s)}\right\}^{1/2} \sum_{l = s+1}^t (X_l - f_l)\Bigg| < b_{t} \Bigg\}.
	\end{align*}	
	It follows from Lemma~\ref{lem-alt-cumsum-prob} that $\mathbb{P}\{\mathcal{B}\} > 1-\alpha$. Throughout the proof we assume that the event $\mathcal{B}$ holds.

For any $s, t \in N$, $1 \leq s < t$, it holds that $\left|\widehat{D}_{s, t} - D_{s, t}\right| < b_{t}$, which implies that 
	\begin{equation}\label{eq-2-22222}
		D_{s, t} + b_{t} > \widehat{D}_{s, t} > D_{s, t} - b_{t}.
	\end{equation}

Step 2.	For any $t \leq \Delta$, we have that $D_{s, t} = 0$, for all $s \in [1, t)$.  Thus, using \eqref{eq-2-22222}, we conclude that, $\widehat{t} > t$ and, therefore that $\widehat{t} > \Delta$.

Step 3.  Now we consider any $t > \Delta$.  If there exists $s \in [1, t)$ such that $\widehat{D}_{s, t} > b_{t}$, then $d \leq t - \Delta$.  Thus, $d \leq \widetilde{t} - \Delta$, where 
	\[
		\widetilde{t} = \min\{t > \Delta, \exists s \in [\Delta, t), \widehat{D}_{s, t} > b_{t}\},
	\]
	and any upper bound on  $\widetilde{t}$ will also be an upper bound on $d$, when the signal-to-noise constraint specified by Assumption~\ref{assump-snr} is satisfied. Thus, our task becomes that of computing a sharp upper bound on  $\widetilde{t}$. To that effect, notice that, when $\Delta \leq s < t$, 
	\[
		D_{s, t} = \Delta \left(\frac{t-s}{ts}\right)^{1/2}|\mu_1 - \mu_2| = \Delta \left(\frac{t-s}{ts}\right)^{1/2} \kappa,
	\]	
	and, because of  \eqref{eq-2-22222} again,
	\[
		\widehat{D}_{s, t} \geq \Delta \left(\frac{t-s}{ts}\right)^{1/2} \kappa - b_{t}.
	\]
	As a result, we obtain that $\widetilde{t} \leq t^*$, where
	\[
		t_* = \min\left\{t > \Delta: \, \max_{s \in [\Delta, t)} \left\{\Delta \left(\frac{t-s}{ts}\right)^{1/2} \kappa - 2b_{t}\right\} \geq 0 \right\}.
	\]

Step 4. Write for convenience  $m = t^* - \Delta$, so that $d \leq m$. Recalling that $b_t = 2^{3/2}\sigma\log^{1/2}(t/\alpha)$, we seek the smallest integer $m$ such that
\[
	\max_{s \in [\Delta, m + \Delta)} \left[\Delta\kappa \left\{\frac{m+\Delta - s}{(m+\Delta)s}\right\}^{1/2}  - 2^{5/2}\sigma\log^{1/2}\{(m+\Delta)/\alpha\}\right] > 0,
\]
which is equivalent to finding the smallest integer $m$ such that
\[
	\max_{s \in [\Delta, m + \Delta)} \left[\Delta^2\kappa^2 - 32\sigma^2\frac{s(m + \Delta)}{m+\Delta-s} \log\left\{(m+\Delta)/\alpha\right\}\right] > 0.
\]
In turn, the above task corresponds to that of computing the smallest integer $m$ such that
\begin{align*}
	\Delta^2\frac{\kappa^2}{\sigma^2} & > \min_{s \in [\Delta, m + \Delta)} \left[32 \frac{s(m + \Delta)}{m+\Delta-s} \log\left\{(m+\Delta)/\alpha\right\}\right] \\
	& = 32 \frac{\Delta(m + \Delta)}{m}\log\left\{(m+\Delta)/\alpha\right\},
\end{align*}
or, equivalently, such that
\begin{equation}\label{eq-m-thm-1}
	m \left[\frac{\Delta\kappa^2}{32\sigma^2} - \log\left\{(m+\Delta)/\alpha\right\}\right] > \Delta  \log\left\{(m+\Delta)/\alpha\right\},
\end{equation}
under Assumption~\ref{assump-snr}.

Let $C_d$ be an absolute constant large enough and also upper bounded by $C_{\mathrm{SNR}}$.
	The claimed result now follows once we show that that the value
	\[
		m^* = \lceil C_d \log(\Delta/\alpha) \sigma^2\kappa^{-2} \rceil
	\] 
	satisfies \eqref{eq-m-thm-1}. To see this, assume for simplicity that $C_d \log(\Delta/\alpha) \sigma^2\kappa^{-2}$ is an integer; if not, the proof only requires trivial modifications. We first point out that $m^* \leq \Delta$ because of Assumption~\ref{assump-snr} and the fact that $C_d \leq C_{\mathrm{SNR}}$. Now, the left hand side of inequality \eqref{eq-m-thm-1} is equal, for this choice of $m$, to
	\begin{equation}\label{eq:step.in.thm1}
C_d \log(\Delta/\alpha) \frac{\Delta}{32} - C_d \frac{\sigma^2}{\kappa^2} \log(\Delta/\alpha) \log \left\{\frac{C_d \log(\Delta/\alpha) \sigma^2 / \kappa^2 + \Delta }{\alpha}\right\}.
	\end{equation}
	Using again Assumption~\ref{assump-snr} and the fact that $m^* \leq \Delta$, the second term in the previous expression is upper bounded by 
	\[
		\frac{2C_d}{C_{\mathrm{SNR}}} \Delta \log(\Delta/\alpha),	
	\]
	due to the fact that $2\log(x) \geq \log(2x)$, $x \geq 2$. Thus, the quantity in \eqref{eq:step.in.thm1} is lower bounded by
\begin{align*}
\Delta \log(\Delta/\alpha)  \left( \frac{C_d}{32} - \frac{2C_d}{C_{\mathrm{SNR}}} \right) & \geq 2\Delta \log(\Delta/\alpha)\geq \Delta \log(2 \Delta/\alpha)\geq \Delta \log\left( (m^* + \Delta)/\alpha \right),
	\end{align*}
where the first inequality is justified by first choosing a large enough $C_d$ and then choosing $C_{\mathrm{SNR}}$ larger than $C_d$, and the second and third inequalities follow from  $\log(\Delta/\alpha) \geq 0$ and  $m^* \leq \Delta$, respectively. Thus, combining the last display with \eqref{eq-m-thm-1} and  \eqref{eq:step.in.thm1} yields  (\ref{eq:thm-2.main}).  Finally, (\ref{eq:thm-2.delta.infty}) and (\ref{eq:thm-2.any.delta})   follow immediately from Steps 1 and 2.  
\end{proof}

\begin{proof}(of Proposition~\ref{prop-undetectable})
It follows from Step 1 in the proof of \Cref{thm-2} that on the event $\mathcal{B}$, it holds that
	\[
		\widehat{D}_{s, t} < D_{s, t} + 2^{3/2}\sigma \log^{1/2}(t/\alpha), \quad 1 \leq s < t.
	\]
	It follows from Step 2 in the proof of \Cref{thm-2} that we only need to consider $t > \Delta$.  This leaves us two situations $s \leq \Delta$ and $s \geq \Delta$.  In fact in both these two situations, one only needs to deal with the case $s = \Delta$, therefore we only show $s \geq \Delta$ here.
	
When $\Delta \leq s < t$, we have $D_{s, t} = \kappa \Delta \{(t-s)/ts\}^{1/2}$ and therefore on the event $\mathcal{B}$, we have that 
	\begin{align*}
		& \max_{\Delta \leq s < t}\widehat{D}_{s, t}  \leq \max_{\Delta \leq s < t} \kappa \Delta \left(\frac{t-s}{ts}\right)^{1/2} + 2^{3/2}\sigma \log^{1/2}(t/\alpha) \\
		\leq &  \kappa\Delta \left(\frac{t-s}{ts}\right)^{1/2} + 2^{3/2}\sigma \log^{1/2}(t/\alpha) \leq \kappa \Delta^{1/2} + 2^{3/2}\sigma \log^{1/2}(t/\alpha) < b_t,
	\end{align*}
	where the last inequality follows from \eqref{eq-low-snr-rrrr}.  We therefore completes the proof.
\end{proof}

\begin{proof}(of Proposition~\ref{prop-alternative-type1})

Step 1. Define the event 
	\begin{align*}
	\mathcal{C} & = \Bigg\{t \in \{2, \ldots, \gamma+2\}, \, s\in N \cap [1, t):\, \Bigg|\left(\frac{t-s}{ts}\right)^{1/2} \sum_{l = 1}^s (X_l - f_l) \\
	& \hspace{2cm} - \left\{\frac{s}{t(t-s)}\right\}^{1/2} \sum_{l = s+1}^t (X_l - f_l)\Bigg| < b_{t} \Bigg\}.
	\end{align*}	
	We have that
	\begin{align*}
		& 1 - \mathbb{P}\{\mathcal{C}\}\\
		\leq & (\gamma+1) \max_{t\in \{2, \ldots, \gamma+2\}} \mathbb{P}\Bigg\{\exists s\in N \cap [1, t):\, \Bigg|\left(\frac{t-s}{ts}\right)^{1/2} \sum_{l = 1}^s (X_l - f_l) \\
		& \hspace{2cm} - \left\{\frac{s}{t(t-s)}\right\}^{1/2} \sum_{l = s+1}^t (X_l - f_l)\Bigg| \geq b_{t} \Bigg\}	 \\
		\leq & (\gamma+1)^2 \mathbb{P}\left\{\left|W\right| > b_{t} \right\} \leq  2(\gamma+1)^2 \exp\left(-\frac{b_{t}^2}{2\sigma^2}\right) = (\gamma + 1)^{-1},
	\end{align*}
	where $W$ is a mean zero sub-Gaussian random variable with $\|W\|_{\psi_2} \leq \sigma$ and the Hoeffding inequality \citep[e.g.~(2.9) in][]{wainwright2019high}.

	Therefore we have that
	\begin{align*}
		& \mathbb{E}_{\infty}(\widehat{t}) = \sum_{t = 2}^{\infty} \mathbb{P}\{\widehat{t} \geq t\} \geq \sum_{t = 2}^{\gamma +  2} \mathbb{P}\{\widehat{t} \geq t\} \geq (\gamma +1 )\mathbb{P}\{\widehat{t} > \gamma + 2\} \geq (\gamma + 1)\left(1 - \frac{1}{\gamma+1}\right)  = \gamma.
	\end{align*} 

Step 2.  We have that
	\[
		\mathbb{P}\{\widehat{t} > \gamma + 2\} = \mathbb{P}\{\mathcal{C}\} \geq 1 - (\gamma + 1)^{-1},
	\]
	When $\gamma > \max\{\Delta - 2, 1\}$, it holds that
	\[
		\mathbb{P}\{d > 0\} \geq 1 - (\gamma+1)^{-1}.
	\]
	For any $t > \Delta$, if there exists $s \in [1, t)$ such that $\widehat{D}_{s, t} > b_{t}$, then $d \leq t - \Delta$.  It suffices to find 
	\[
		\widetilde{t} = \min\{t: \, t > \Delta, \exists s \in (\Delta, t), \widehat{D}_{s, t} > b_{t}\},
	\]
	and any upper bound on $\widetilde{t}$ will also be an upper bound on $d$, when the signal-to-noise ratio condition in \eqref{eq-m-gamma-cond} is satisfied.  Then, our task becomes that of computing a sharp upper bound on $\widetilde{t}$.  To that effect, notice that, when $\Delta \leq s < t$, 
	\[
		D_{s, t} = \Delta \left(\frac{t-s}{ts}\right)^{1/2}|\mu_1 - \mu_2| = \Delta \left(\frac{t-s}{ts}\right)^{1/2}\kappa,
	\]	
	and on the event $\mathcal{C}$,
	\[
		\widehat{D}_{s, t} \geq \Delta \left(\frac{t-s}{ts}\right)^{1/2}\kappa - b_{t}.
	\]
	As a result, we obtain that $\widetilde{t} \leq t^*$, where
	\[
		t_* = \min\left\{t > \Delta: \, \max_{s \in (\Delta, t)} \left\{\Delta \left(\frac{t-s}{ts}\right)^{1/2} \kappa - 2b_{t}\right\} \geq 0 \right\}.
	\]

Step 3.  Write for convenience $m = t - \Delta$, so that $d \leq m$.  Recalling that 
	\[
		b_t = 6^{1/2}\sigma\log^{1/2}\left\{2^{1/3} (\gamma + 1)\right\},
	\] 
	we seek the smallest integer $m$ such that 
\[
	\max_{s \in [\Delta, m + \Delta)} \left[\Delta\kappa \left\{\frac{m+\Delta - s}{(m+\Delta)s}\right\}^{1/2}  - 2\times 6^{1/2}\sigma\log^{1/2}\left\{2^{1/3} (\gamma + 1)\right\}\right] > 0,
\]
which is equivalent to finding the smallest integer $m$ such that
\begin{equation*}
	\max_{s \in [\Delta, m + \Delta)} \left\{\Delta^2\kappa^2 - 24\sigma^2\frac{s(m + \Delta)}{m+\Delta-s} \log\left\{2^{1/3} (\gamma + 1)\right\}\right\} > 0.
\end{equation*}
In turn, the above task corresponds to that of computing the smallest integer $m$ such that 
\begin{align*}t
	\Delta^2\kappa^2 & > \min_{s \in [\Delta, m + \Delta)} \left\{24\sigma^2\frac{s(m + \Delta)}{m+\Delta-s} \log\left\{2^{1/3} (\gamma + 1)\right\}\right\} \\
	& = 24\sigma^2\frac{\Delta(m + \Delta)}{m} \log\left\{2^{1/3} (\gamma + 1)\right\},
\end{align*}
or, equivalently, such that 
\begin{equation}\label{eq-aaaaaaaaaaaaaaaa}
	m \left[\frac{\Delta\kappa^2}{24\sigma^2} - \log\left\{2^{1/3} (\gamma + 1)\right\}\right] > \Delta  \log\left\{2^{1/3} (\gamma + 1)\right\}.
\end{equation}

Let $C_d'$ be a large enough absolute constant.  The claimed result now follows once we show that the value
\[
	m^* = \Big\lceil C_d' \sigma^2 \kappa^{-2}\log\left\{2^{1/3} (\gamma + 1)\right\}\Big\rceil
\]
satisfies \eqref{eq-aaaaaaaaaaaaaaaa}.  To see this, assume for simplicity that $C_d' \sigma^2 \kappa^{-2}\log\left\{2^{1/3} (\gamma + 1)\right\}$ is an integer; if not, the proof only requires trivial modifications.  Now, the left hand side of inequality \eqref{eq-aaaaaaaaaaaaaaaa} is equal, for this choice of $m$, to
	\begin{equation}\label{eq:step.in.thm111111}
C_d' \log\left\{2^{1/3} (\gamma + 1)\right\} \frac{\Delta}{24} - C_d'\sigma^2\kappa^{-2}\log^2\left\{2^{1/3} (\gamma + 1)\right\}.
	\end{equation}
	Using again \eqref{eq-m-gamma-cond}, the second term in the previous expression is upper bounded by 
	\[
\frac{2C_d'}{C_{\mathrm{SNR}}} \Delta \log\left\{2^{1/3} (\gamma + 1)\right\},
	\]	
	due to the fact that $2\log(x) \geq \log\{2^{1/3}(x+1)\}$, $x \geq 2$.  Thus, the quantity in \eqref{eq:step.in.thm111111} is lower bounded by
\begin{align*}
\Delta \log\left\{2^{1/3} (\gamma + 1)\right\} \left( \frac{C_d'}{24} - \frac{2C_d'}{C_{\mathrm{SNR}}} \right) & \geq 2\Delta \log\left\{2^{1/3} (\gamma + 1)\right\} \geq \Delta \log\left\{2^{1/3} (\gamma + 1)\right\},
	\end{align*}
where the first inequality is justified with first choosing a large enough $C_d$, then a large enough $C_{\mathrm{SNR}}$.  Thus we conclude the proof with $C_d = 2C_d'$.
\end{proof}

\begin{proof}(of Proposition~\ref{prop-alter-2})
When $\Delta = \infty$, assuming that $f_1 = f_2 = \ldots$ and letting $e_{\gamma, t} = \max\{t-\gamma+1,\, 1\}$, we have that
	\begin{align*}
		& \sup_{v > 1} \mathbb{P}(v \leq \widehat{t} \leq v + \gamma) \\
		\leq & \sup_{v > 1} \mathbb{P}\left\{\exists s \in N \cap (v, v + \gamma), \, t \in N \cap [e_{\gamma, t}, t-1]: \, \widehat{D}_{e_{\gamma, t}, s, t} > b_{\gamma}\right\} \\
		\leq & \gamma^2 \mathbb{P}\left\{\left|W\right| > b_{\gamma}\right\} \leq \alpha,
	\end{align*}
	where $W$ is a mean zero sub-Gaussian random variable with $\|W\|_{\psi_2} \leq \sigma$ and satisfies that any $\zeta > 0$, $\mathbb{P} \left\{\left|W\right| \geq \zeta \right\} < 2\exp(-2^{-1}\zeta^2/\sigma^2)$.  
		
When Assumption~\ref{assump-2} is imposed, we consider $t > \Delta$.  For any $\Delta < t < \Delta + \gamma/2$, if there exists $s \in [1, t)$ such that $\widehat{D}_{e_{\gamma}, s, t} > b_{\gamma}$, then $d \leq t - \Delta$.  It suffices to find 
	\[
		\widetilde{t} = \min\{t: \, \Delta < t < \Delta + \gamma/2, \widehat{D}_{t-\gamma, \Delta, t} > b_{\gamma}\},
	\]
	and any upper bound on $\widetilde{t}$ will also be an upper bound on $d$, when \eqref{eq-alter-gamma-cond} is satisfied.  Thus, our task becomes that of computing a sharp upper bound on $\widetilde{t}$.  To that effect, notice that, when $\Delta = s < t < \Delta + \gamma/2$, we have 
	\[
		D_{t - \gamma, \Delta, t} = \left\{\frac{(t-\Delta)(\Delta -t + \gamma)}{\gamma}\right\}^{1/2}|\mu_1 - \mu_2| = \left\{\frac{(t-\Delta)(\Delta -t + \gamma)}{\gamma}\right\}^{1/2}\kappa.
	\]	
	In addition, it holds that
	\begin{align*}
		\mathbb{P}\{\mathcal{G}\} = \mathbb{P}\left\{\exists t \in (\Delta, \Delta + \gamma/2): \, |\widehat{D}_{t-\gamma, \Delta, t} - D_{t-\gamma, \Delta, t}| > b_{\gamma}\right\} \leq \gamma/2 \alpha \gamma^{-2} < \alpha.
	\end{align*}
	Therefore on the event $\mathcal{G}$,
	\[
		\widehat{D}_{t-\gamma, \Delta, t} \geq \left\{\frac{(t-\Delta)(\Delta -t + \gamma)}{\gamma}\right\}^{1/2}\kappa - b_{\gamma}.
	\]
	As a result, we obtain that $\widetilde{t} \leq t^*$, where
	\[
		t_* = \min\left\{t \in (\Delta, \Delta + \gamma/2): \, \left\{\frac{(t-\Delta)(\Delta -t + \gamma)}{\gamma}\right\}^{1/2}\kappa - 2b_{\gamma} \geq 0 \right\}.
	\]

Write for convenience $m = t_* - \Delta$, so that $d \leq m$.  Due to the form of $b_{\gamma}$, we seek the smallest integer $m$ such that
	\[
		\left\{\frac{m(\gamma - m)}{\gamma}\right\}^{1/2}\kappa - 2^{3/2}\sigma\log^{1/2}\left(2\gamma^2\alpha^{-1}\right) > 0.
	\]
	We now to show that $m = \lceil C_d \sigma^2 \log(\gamma/\alpha) \kappa^{-2} \rceil$ satisfies the above.  Due to \eqref{eq-alter-gamma-cond}, we have that $(\gamma - m)/\gamma \geq 1/2$, then 
	\begin{align*}
		\frac{m(\gamma - m)}{\gamma}\kappa^2 \geq \kappa^2 m/2 >  2^{3/2}\sigma\log^{1/2}\left(2\gamma^2\alpha^{-1}\right),
	\end{align*}
	which completes the proof.
\end{proof}

\begin{proof}(of Corollary~\ref{prop-new})
\noindent Step 1. Define the event 
	\begin{align*}
		& \mathcal{B} = \Bigg\{\forall s, t \in N,\, t > 1, \, s \in \{t-2^0, \ldots, t-2^{\lfloor\log(t)/\log(2)\rfloor}\}: \\
		& \hspace{2cm} \left|\left(\frac{t-s}{ts}\right)^{1/2} \sum_{l = 1}^s (X_l - f_l) - \left\{\frac{s}{t(t-s)}\right\}^{1/2} \sum_{l = s+1}^t (X_l - f_l)\right| < b_{t} \Bigg\}.
	\end{align*}	
	It follows from Lemma~\ref{lem-alt-cumsum-prob} that $\mathbb{P}\{\mathcal{B}\} > 1-\alpha$.  Throughout the proof we assume that the event $\mathcal{B}$ holds.

For any $s, t \in N$, $1 \leq s < t$, it holds that $\left|\widehat{D}_{s, t} - D_{s, t}\right| < b_{t}$, which implies that 
	\begin{equation}\label{eq-2-2}
		D_{s, t} + b_{t} > \widehat{D}_{s, t} > D_{s, t} - b_{t}.
	\end{equation}

For any $t \geq 1$, we have that 
	\[
		\mathcal{S}(t) = \{t-2^0, \ldots, t-2^{\lfloor\log(t)/\log(2)\rfloor}\}.
	\]
	
Step 2.	For any $t \leq \Delta$, we have that $D_{s, t} = 0$, for all $s \in \mathcal{S}(t)$.  Using \eqref{eq-2-2}, we conclude that $\widehat{t} > \Delta$ and therefore that $\widehat{t} > \Delta$.

Step 3.  Now we consider $t > \Delta$.  If there exists $s \in \mathcal{S}(t)$ such that $\widehat{D}_{s, t} > b_{t}$, then $d \leq t - \Delta$.  Thus, $d \leq \widetilde{t} - \Delta$, where
	\[
		\widetilde{t} = \min\{t: \, t > \Delta, \exists s \in [\Delta, t) \cap  \mathcal{S}(t), \widehat{D}_{s, t} > b_{t}\},
	\]
	and any upper bound on $\widetilde{t}$ will also be an upper bound on $d$, when the signal-to-noise ratio constraint specified by Assumption~\ref{assump-snr} is satisfied.  Thus, our task becomes that of computing a sharp upper bound on $\widetilde{t}$.  To that effect, notice that, when $\Delta \leq s < t$, 
	\[
		D_{s, t} = \Delta \left(\frac{t-s}{ts}\right)^{1/2}|\mu_1 - \mu_2| = \Delta \left(\frac{t-s}{ts}\right)^{1/2} \kappa,
	\]	
	and, because of \eqref{eq-2-2} again,
	\[
		\widehat{D}_{s, t} \geq \Delta \left(\frac{t-s}{ts}\right)^{1/2} \kappa - b_{t}.
	\]
	As a result, we obtain that $\widetilde{t} \leq t^*$, where
	\[
		t_* = \min\left\{t > \Delta: \, \max_{s \in [\Delta, t) \cap  \mathcal{S}(t)} \left\{\Delta \left(\frac{t-s}{ts}\right)^{1/2} \kappa - 2b_{t}\right\} \geq 0 \right\}.
	\]

Step 3.  Write for convenience $m = t - \Delta$, so that $d \leq m$.  Recalling that $b_t = 2^{3/2}\sigma\log^{1/2}(t/\alpha)$, we seek the smallest integer $m$ such that
\[
	\max_{s \in [\Delta, m + \Delta) \cap  \mathcal{S}(m + \Delta)} \left[\Delta\kappa \left\{\frac{m+\Delta - s}{(m+\Delta)s}\right\}^{1/2}  - 2^{5/2} \sigma\log^{1/2}\{(m+\Delta)/\alpha\}\right] > 0,
\]
which is equivalent to 
\[
	\max_{s \in [\Delta, m + \Delta) \cap  \mathcal{S}(m + \Delta)} \left[\Delta^2\kappa^2 - 32\sigma^2\frac{s(m + \Delta)}{m+\Delta-s} \log\left\{(m+\Delta)/\alpha\right\}\right] > 0.
\]
In turn, the above task corresponds to that of computing the smallest integer $m$ such that
\begin{align*}
	\Delta^2\kappa^2 > \min_{s \in [\Delta, m + \Delta) \cap  \mathcal{S}(m + \Delta)} \left[32\sigma^2\frac{s(m + \Delta)}{m+\Delta-s} \log\left\{(m+\Delta)/\alpha\right\}\right].
\end{align*}

Since $[\Delta, m + \Delta) \cap  \mathcal{S}(m + \Delta) = \{m+\Delta - 2^j, \, j = 0, \ldots, \lfloor \log(m)/\log(2)\rfloor\}$, in the following we consider two cases.

\medskip
\noindent Case 1.  If $\log(m)/\log(2)$ is an integer, then it follows from identical arguments in the proof of \Cref{thm-2} that
\begin{align*}
	\Delta^2\kappa^2 & > \min_{s \in [\Delta, m + \Delta) \cap  \mathcal{S}(t)} \left[32\sigma^2\frac{s(m + \Delta)}{m+\Delta-s} \log\left\{(m+\Delta)/\alpha\right\}\right] \\
	& = 32\sigma^2\frac{\Delta(m + \Delta)}{m}\log\left\{(m+\Delta)/\alpha\right\}.
\end{align*}

\medskip
\noindent Case 2.  If $\log(m)/\log(2)$ is not an integer, then $\lfloor \log(m)/\log(2)\rfloor < \log(m)/\log(2)$.  This means
\begin{align*}
	\Delta^2\kappa^2 & > \min_{s \in [\Delta, m + \Delta) \cap  \mathcal{S}(t)} \left[32\sigma^2\frac{s(m + \Delta)}{m+\Delta-s} \log\left\{(m+\Delta)/\alpha\right\}\right] \\
	& > 32\sigma^2\frac{\Delta(m + \Delta)}{m}\log\left\{(m+\Delta)/\alpha\right\}.
\end{align*}

Both cases lead to finding the smallest integer $m$ such that
	\[
		m \left[\frac{\Delta\kappa^2}{32\sigma^2} - \log\left\{(m+\Delta)/\alpha\right\}\right] > \Delta  \log\left\{(m+\Delta)/\alpha\right\}.
	\]
	It follows from the identical arguments in the proof of \Cref{thm-2} that we complete the proof.
\end{proof}

\begin{proof}(of \Cref{thm-multi})
Step 1. To prove the theorem, it is sufficient (in fact, equivalent) to show that the claim holds true simultaneously over all deterministic stopping times. Towards that end, define the event 
	\begin{align*}
	& \mathcal{E} = \Bigg\{\forall e, s, t \in \mathbb{N},\, e \geq 0, \, s > e, \, t > s:\, \Bigg|\left\{\frac{t-s}{(s-e)(t-e)}\right\}^{1/2} \sum_{l =e+1}^s (X_l - f_l) \\
	& \hspace{2cm} - \left\{\frac{s-e}{(t-s)(t-e)}\right\}^{1/2} \sum_{l = s+1}^t (X_l - f_l)\Bigg| <b_t \Bigg\}.
	\end{align*}	
	Then, under $\mathcal{E}$, the fluctuations of the CUSUM process are controlled uniformly at any time $t$, including any stopping (in fact, more generally, random) time $T$.
	It follows from Lemma~\ref{lem-concen-multi} that 
	\[
		\mathbb{P}\{\mathcal{E}\} > 1-\alpha.
	\]
	On the event $\mathcal{E}$, for any $0\leq e < s < t$, it holds that
	\[
		\left|\widehat{D}_{e, s, t} - D_{e, s, t}\right| < b_t,
	\]
	which implies that 
	\begin{equation}\label{eq-2-multi}
		D_{e, s, t} + b_{e, t} > \widehat{D}_{e, s, t} > D_{e, s, t} - b_{e, t}.
	\end{equation}

In addition, due to Assumption~\ref{assump-snr-2}, we have that 
	\begin{equation}\label{eq-error-delta}
		C_d\frac{\sigma^2 \log(\Delta/\alpha)}{\kappa^2} \leq \Delta/4.
	\end{equation}
	Then it suffices to show that
	\begin{enumerate}
	\item [(i)]	for any refresh starting point of the algorithm $e$ and any interval $(e, t]$ not containing any true change points, on the event $\mathcal{E}$, there is no detected change point;
	\item [(ii)] on the event $\mathcal{E}$, we can detect $\eta_k$ with delay upper bounded by $C_d \sigma^2 \log(\Delta/\alpha) \kappa_k^{-2}$.
	\end{enumerate}

Step 2.	As for (i), it holds automatically due to the definition of the event $\mathcal{E}$.  The claim (i) leads to that $\widehat{t}_k > \eta_k$.

Step 3.	As for (ii), we prove by induction.  When $k = 0$, we have $\widehat{t}_k = \eta_k = 0$, then $\eta_1 - \widehat{t}_0 \geq \Delta \geq 3\Delta/4$.  It follows from identical arguments in the proof of \Cref{thm-2} that 
	\[
		d_1 = \widehat{t}_1 - \eta_1 \leq C_d \frac{\sigma^2 \log(\Delta/\alpha)}{\kappa_1^2}.
	\]
	Due to \eqref{eq-error-delta}, we have that $\eta_2 - \widehat{t}_1 \geq 3\Delta/4$, then due to \Cref{alg-online-multi}, the procedure restarts by setting $e = \widehat{t}_1$.  For a general $k \geq 1$, if $\eta_k - \widehat{t}_{k-1} \geq 3\Delta/4$, then it follows from  from identical arguments in the proof of \Cref{thm-2} that 
	\[
		d_k = \widehat{t}_k - \eta_k \leq C_d \frac{\sigma^2 \log(\Delta/\alpha)}{\kappa_k^2},
	\]
	which completes the proof.

\end{proof}

\begin{proof}(of Proposition~\ref{prop-lb})
Throughout the proof we will assume for simplicity that $\frac{\sigma^2}{2\kappa^2} \log\left(\frac{1}{\alpha}\right)$ is an integer.
Step 1.  For any $n$, let $P^n$ be the restrictions of a distribution $P$ to $\mathcal{F}_n$, i.e.~the $\sigma$-field generated by the observations $\{X_i\}_{i = 1}^n$.  For any $\nu \geq 1$ and $n \geq \nu$, we have that for any $n \geq \Delta$, it holds that
	\begin{align*}
		\frac{dP_{\kappa, \sigma, \nu}^ n}{dP_{\kappa, \sigma, \infty}^n} = \exp\left(\sum_{i = \nu + 1}^n Z_i\right),
	\end{align*}
	where $P_{\kappa, \sigma, \infty}$ indicates the joint distribution under which there is no change point and
	\[
		Z_i = \frac{\mu_2 - \mu_1}{\sigma^2} \left(X_i - \frac{\mu_1 + \mu_2}{2}\right).
	\]

	For any $\nu \geq 1$, define the event
	\[
		\mathcal{E}_{\nu} = \left\{\nu < T < \nu + \frac{\sigma^2}{2\kappa^2} \log\left(\frac{1}{\alpha}\right), \, \sum_{i = \nu + 1}^T Z_i < \frac{3}{4}\log\left(\frac{1}{\alpha}\right)\right\}.
	\]
	Then we have
	\begin{align}
		& \mathbb{P}_{\kappa, \sigma, \nu}(\mathcal{E}_{\nu}) = \int_{\mathcal{E}_{\nu}} \exp\left(\sum_{i = \nu + 1}^T Z_i\right) \, dP_{\kappa, \sigma, \infty} \leq \exp\left\{(3/4) \log(1/\alpha)\right\} \mathbb{P}_{\kappa, \sigma, \infty}(\mathcal{E}_{\nu}) \nonumber \\
		\leq & \exp\left\{(3/4) \log(1/\alpha)\right\} \mathbb{P}_{\kappa, \sigma, \infty}\left\{\nu < T < \nu + \frac{\sigma^2}{2\kappa^2} \log\left(\frac{1}{\alpha}\right)\right\} \leq \alpha^{-3/4}\alpha = \alpha^{1/4}, \label{eq-lb-2-leq}
	\end{align}
	where the first two inequalities follow from the definition of $\mathcal{E}_{\nu}$, and the last inequality follows from the definition of $\mathcal{D}(\alpha)$.
	 
Step 2.  For any $\nu \geq 1$ and $T \in \mathcal{D}(\alpha)$, since $\{T \geq \nu\} \in \mathcal{F}_{\nu - 1}$, we have that
	\begin{align*}
	& \mathbb{P}_{\kappa, \sigma, \nu} \left\{\nu < T < \nu + \frac{\sigma^2}{2\kappa^2}\log\left(\frac{1}{\alpha}\right), \, \sum_{i = \nu + 1}^T Z_i \geq (3/4)\log(1/\alpha) \, \Big | \, T > \nu \right\}	 \\
	\leq & \esssup \mathbb{P}_{\kappa, \sigma, \nu} \left\{\max_{1 \leq t \leq \frac{\sigma^2}{2\kappa^2}\log\left(\frac{1}{\alpha}\right) - 1} \sum_{i = \nu + 1}^{\nu + t} Z_i \geq (3/4) \log(1/\alpha) \, \Big | X_1, \ldots, X_{\nu}\right\} \\
	\leq & \esssup \mathbb{P}_{\kappa, \sigma, \nu} \Bigg[\max_{1 \leq t \leq \frac{\sigma^2}{2\kappa^2}\log\left(\frac{1}{\alpha}\right)-1} \sum_{i = \nu + 1}^{\nu + t} \{Z_i - \kappa^2/(2\sigma^2)\} \geq (3/4) \log(1/\alpha) \\
	& \hspace{4cm} - \frac{\sigma^2}{2\kappa^2} \log\left(\frac{1}{\alpha}\right)\frac{\kappa^2}{2\sigma^2} \Big | X_1, \ldots, X_{\nu}\Bigg] \\
	\leq & \esssup \mathbb{P}_{\kappa, \sigma, \nu} \Bigg[\max_{1 \leq t \leq \frac{\sigma^2}{2\kappa^2}\log\left(\frac{1}{\alpha}\right)-1} \sum_{i = \nu + 1}^{\nu + t} \{Z_i - \kappa^2/(2\sigma^2)\} \geq (1/2) \log(1/\alpha)\, \Big | X_1, \ldots, X_{\nu}\Bigg] \\
	\leq & \frac{\sigma^2}{2\kappa^2}\log\left(\frac{1}{\alpha}\right) \exp \left\{-\frac{(1/2) \log^2(1/\alpha)}{\frac{\sigma^2}{2\kappa^2}\log\left(\frac{1}{\alpha}\right) \frac{\kappa^2}{\sigma^2}}\right\} = \frac{\sigma^2}{2\kappa^2}\log\left(\frac{1}{\alpha}\right) \exp \left\{- \log(1/\alpha)\right\} \leq  \alpha^{1/4},
	\end{align*}
	where the fourth inequality follows from the Hoeffding inequality and a union bound argument, and the last inequality holds for small enough $\alpha$ such that 
	\[
		\alpha^{5/4}\log(1/\alpha) \leq 2\kappa^2\sigma^{-2}.
	\]  
	Since the upper bound is independent of $\nu$, it holds that
	\[
		\sup_{\nu \geq 1}\mathbb{P}_{\kappa, \sigma, \nu} \left\{\nu < T < \nu + \frac{\sigma^2}{2\kappa^2}\log\left(\frac{1}{\alpha}\right), \, \sum_{i = \nu + 1}^T Z_i \geq (3/4)\log(1/\alpha) \, \Big | \, T \geq \nu \right\}	\leq \alpha^{1/4},
	\]
	which leads to 
	\begin{align}\label{eq-lb-1-geq}
		\sup_{\nu \geq 1}\mathbb{P}_{\kappa, \sigma, \nu} \left\{\nu < T < \nu + \frac{\sigma^2}{2\kappa^2}\log\left(\frac{1}{\alpha}\right), \, \sum_{i = \nu + 1}^T Z_i \geq (3/4)\log(1/\alpha) \right\} \leq \alpha^{1/4}
	\end{align}
	Combining \eqref{eq-lb-2-leq} and \eqref{eq-lb-1-geq}, we have	
	\begin{equation}\label{eq-lb-1-comb}
		\sup_{\nu \geq 1}\mathbb{P}_{\kappa, \sigma, \nu} \left\{\nu < T < \nu + \frac{\sigma^2}{2\kappa^2}\log\left(\frac{1}{\alpha}\right)\right\} \leq 2\alpha^{1/4}.
	\end{equation}

Step 3.  We now have, for any change point time $\Delta$,
	\begin{align*}
		 & \mathbb{E}_{\kappa, \sigma, \Delta} \left\{(T - \Delta)_+	\right\} \geq \frac{\sigma^2}{2\kappa^2} \log(1/\alpha) \mathbb{P}_{\kappa, \sigma, \Delta}\left\{T - \Delta \geq \frac{\sigma^2}{2\kappa^2} \log(1/\alpha)\right\} \\
		= & \frac{\sigma^2}{2\kappa^2} \log(1/\alpha)\left[\mathbb{P}_{\kappa, \sigma, \Delta}\{T > \Delta\} - \mathbb{P}_{\kappa, \sigma, \Delta}\left\{\Delta < T  < \Delta + \frac{\sigma^2}{2\kappa^2}\log\left(\frac{1}{\alpha}\right) \right\}\right] \\
		\geq & \frac{\sigma^2}{2\kappa^2} \log(1/\alpha) (1 - \alpha - 2\alpha^{1/4}) \geq \frac{\sigma^2}{4\kappa^2} \log(1/\alpha),
	\end{align*}
	where the first inequality is due to Markov's inequality, the second is due to \eqref{eq-lb-1-comb} and the definition of the class of $\mathcal{D}(\alpha)$ of stopping times (which in particular implies that $\mathbb{P}_{\kappa, \sigma, \Delta}\{T \leq \Delta\} = \mathbb{P}_\infty(T \leq \Delta) \leq \mathbb{P}_\infty(T < \infty) \leq \alpha$), and the last holds when $\alpha + 2\alpha^{1/4} < 1/2$. 
\end{proof}

\bibliographystyle{ims}
\bibliography{ref}

\begin{thebibliography}{34}
\expandafter\ifx\csname natexlab\endcsname\relax\def\natexlab#1{#1}\fi
\expandafter\ifx\csname url\endcsname\relax
  \def\url#1{\texttt{#1}}\fi
\expandafter\ifx\csname urlprefix\endcsname\relax\def\urlprefix{URL }\fi
\providecommand{\eprint}[2][]{\url{#2}}

\bibitem[{Aue and Horv{\'a}th(2004)}]{aue2004delay}
\textsc{Aue, A.} and \textsc{Horv{\'a}th, L.} (2004).
\newblock Delay time in sequential detection of change.
\newblock \textit{Statistics \& Probability Letters}, \textbf{67} 221--231.

\bibitem[{Aue et~al.(2009)Aue, Horv{\'a}th and Reimherr}]{aue2009delay}
\textsc{Aue, A.}, \textsc{Horv{\'a}th, L.} and \textsc{Reimherr, M.~L.} (2009).
\newblock Delay times of sequential procedures for multiple time series
  regression models.
\newblock \textit{Journal of Econometrics}, \textbf{149} 174--190.

\bibitem[{Chen et~al.(2009)Chen, Atev and Lerman}]{chen2009kernel}
\textsc{Chen, G.}, \textsc{Atev, S.} and \textsc{Lerman, G.} (2009).
\newblock Kernel spectral curvature clustering (kscc).
\newblock In \textit{Computer Vision Workshops (ICCV Workshops), 2009 IEEE 12th
  International Conference on}. IEEE, 765--772.

\bibitem[{Chen(2019)}]{chen2019sequential}
\textsc{Chen, H.} (2019).
\newblock Sequential change-point detection based on nearest neighbors.
\newblock \textit{The Annals of Statistics}, \textbf{47} 1381--1407.

\bibitem[{Chen et~al.(2020)Chen, Wang and Samworth}]{chen2020high}
\textsc{Chen, Y.}, \textsc{Wang, T.} and \textsc{Samworth, R.~J.} (2020).
\newblock High-dimensional, multiscale online changepoint detection.
\newblock \textit{arXiv preprint arXiv:2003.03668}.

\bibitem[{Chu et~al.(1996)Chu, Stinchcombe and White}]{chu1996monitoring}
\textsc{Chu, C.-S.~J.}, \textsc{Stinchcombe, M.} and \textsc{White, H.} (1996).
\newblock Monitoring structural change.
\newblock \textit{Econometrica: Journal of the Econometric Society} 1045--1065.

\bibitem[{Desobry et~al.(2005)Desobry, Davy and Doncarli}]{desobry2005online}
\textsc{Desobry, F.}, \textsc{Davy, M.} and \textsc{Doncarli, C.} (2005).
\newblock An online kernel change detection algorithm.
\newblock \textit{IEEE Trans. Signal Processing}, \textbf{53} 2961--2974.

\bibitem[{Dette and G{\"o}smann(2019)}]{dette2019likelihood}
\textsc{Dette, H.} and \textsc{G{\"o}smann, J.} (2019).
\newblock A likelihood ratio approach to sequential change point detection for
  a general class of parameters.
\newblock \textit{Journal of the American Statistical Association} 1--17.

\bibitem[{Fearnhead and Liu(2007)}]{fearnhead2007line}
\textsc{Fearnhead, P.} and \textsc{Liu, Z.} (2007).
\newblock On-line inference for multiple changepoint problems.
\newblock \textit{Journal of the Royal Statistical Society: Series B
  (Statistical Methodology)}, \textbf{69} 589--605.

\bibitem[{G{\"o}smann et~al.(2019)G{\"o}smann, Kley and Dette}]{gosmann2019new}
\textsc{G{\"o}smann, J.}, \textsc{Kley, T.} and \textsc{Dette, H.} (2019).
\newblock A new approach for open-end sequential change point monitoring.
\newblock \textit{arXiv preprint arXiv:1906.03225}.

\bibitem[{He et~al.(2018)He, Xie, Wu and Lin}]{he2018sequential}
\textsc{He, X.}, \textsc{Xie, Y.}, \textsc{Wu, S.-M.} and \textsc{Lin, F.-C.}
  (2018).
\newblock Sequential graph scanning statistic for change-point detection.
\newblock In \textit{2018 52nd Asilomar Conference on Signals, Systems, and
  Computers}. IEEE, 1317--1321.

\bibitem[{Hl{\'a}vka et~al.(2016)Hl{\'a}vka, Hu{\v{s}}kov{\'a}, Kirch and
  Meintanis}]{hlavka2016bootstrap}
\textsc{Hl{\'a}vka, Z.}, \textsc{Hu{\v{s}}kov{\'a}, M.}, \textsc{Kirch, C.} and
  \textsc{Meintanis, S.~G.} (2016).
\newblock Bootstrap procedures for online monitoring of changes in
  autoregressive models.
\newblock \textit{Communications in Statistics-Simulation and Computation},
  \textbf{45} 2471--2490.

\bibitem[{Howard et~al.(2018)Howard, Ramdas, McAuliffe and
  Sekhon}]{howard2018uniform}
\textsc{Howard, S.~R.}, \textsc{Ramdas, A.}, \textsc{McAuliffe, J.} and
  \textsc{Sekhon, J.} (2018).
\newblock Uniform, nonparametric, non-asymptotic confidence sequences.
\newblock \textit{arXiv preprint arXiv:1810.08240}.

\bibitem[{Hu{\v{s}}kov{\'a} and Kirch(2012)}]{huvskova2012bootstrapping}
\textsc{Hu{\v{s}}kov{\'a}, M.} and \textsc{Kirch, C.} (2012).
\newblock Bootstrapping sequential change-point tests for linear regression.
\newblock \textit{Metrika}, \textbf{75} 673--708.

\bibitem[{Hu{\v{s}}kov{\'a} et~al.(2010)Hu{\v{s}}kov{\'a}, Kirch and
  Meintanis}]{huvskova2010fourier}
\textsc{Hu{\v{s}}kov{\'a}, M.}, \textsc{Kirch, C.} and \textsc{Meintanis,
  S.~G.} (2010).
\newblock Fourier methods for sequential change point analysis in
  autoregressive models.
\newblock In \textit{Proceedings of COMPSTAT'2010}. Springer, 501--508.

\bibitem[{Keshavarz et~al.(2018)Keshavarz, Michailidis and
  Atchade}]{keshavarz2018sequential}
\textsc{Keshavarz, H.}, \textsc{Michailidis, G.} and \textsc{Atchade, Y.}
  (2018).
\newblock Sequential change-point detection in high-dimensional gaussian
  graphical models.
\newblock \textit{arXiv preprint arXiv:1806.07870}.

\bibitem[{Kirch(2008)}]{kirch2008bootstrapping}
\textsc{Kirch, C.} (2008).
\newblock Bootstrapping sequential change-point tests.
\newblock \textit{Sequential Analysis}, \textbf{27} 330--349.

\bibitem[{Kirch and Weber(2018)}]{kirch2018modified}
\textsc{Kirch, C.} and \textsc{Weber, S.} (2018).
\newblock Modified sequential change point procedures based on estimating
  functions.
\newblock \textit{Electronic Journal of Statistics}, \textbf{12} 1579--1613.

\bibitem[{Kurt et~al.(2018)Kurt, Yilmaz and Wang}]{kurt2018real}
\textsc{Kurt, M.~N.}, \textsc{Yilmaz, Y.} and \textsc{Wang, X.} (2018).
\newblock Real-time nonparametric anomaly detection in high-dimensional
  settings.
\newblock \textit{arXiv preprint arXiv:1809.05250}.

\bibitem[{Lai(1981)}]{lai1981asymptotic}
\textsc{Lai, T.~L.} (1981).
\newblock Asymptotic optimality of invariant sequential probability ratio
  tests.
\newblock \textit{The Annals of Statistics} 318--333.

\bibitem[{Lai(1995)}]{lai1995sequential}
\textsc{Lai, T.~L.} (1995).
\newblock Sequential changepoint detection in quality control and dynamical
  systems.
\newblock \textit{Journal of the Royal Statistical Society: Series B
  (Methodological)}, \textbf{57} 613--644.

\bibitem[{Lai(1998)}]{lai1998information}
\textsc{Lai, T.~L.} (1998).
\newblock Information bounds and quick detection of parameter changes in
  stochastic systems.
\newblock \textit{IEEE Transactions on Information Theory}, \textbf{44}
  2917--2929.

\bibitem[{Lai(2001)}]{lai2001sequential}
\textsc{Lai, T.~L.} (2001).
\newblock Sequential analysis: some classical problems and new challenges.
\newblock \textit{Statistica Sinica} 303--351.

\bibitem[{Lorden(1971)}]{lorden1971procedures}
\textsc{Lorden, G.} (1971).
\newblock Procedures for reacting to a change in distribution.
\newblock \textit{The Annals of Mathematical Statistics}, \textbf{42}
  1897--1908.

\bibitem[{Maillard(2019)}]{maillard2019sequential}
\textsc{Maillard, O.-A.} (2019).
\newblock Sequential change-point detection: Laplace concentration of scan
  statistics and non-asymptotic delay bounds.
\newblock In \textit{Algorithmic Learning Theory}. 610--632.

\bibitem[{Mei(2010)}]{mei2010efficient}
\textsc{Mei, Y.} (2010).
\newblock Efficient scalable schemes for monitoring a large number of data
  streams.
\newblock \textit{Biometrika}, \textbf{97} 419--433.

\bibitem[{Moustakides(1986)}]{moustakides1986optimal}
\textsc{Moustakides, G.~V.} (1986).
\newblock Optimal stopping times for detecting changes in distributions.
\newblock \textit{The Annals of Statistics}, \textbf{14} 1379--1387.

\bibitem[{Namoano et~al.(2019)Namoano, Starr, Emmanouilidis and
  Cristobal}]{namoano2019online}
\textsc{Namoano, B.}, \textsc{Starr, A.}, \textsc{Emmanouilidis, C.} and
  \textsc{Cristobal, R.~C.} (2019).
\newblock Online change detection techniques in time series: An overview.
\newblock In \textit{2019 IEEE International Conference on Prognostics and
  Health Management (ICPHM)}. IEEE, 1--10.

\bibitem[{Page(1954)}]{Page1954}
\textsc{Page, E.~S.} (1954).
\newblock Continuous inspection schemes.
\newblock \textit{Biometrika}, \textbf{41} 100--115.

\bibitem[{Ritov(1990)}]{ritov1990decision}
\textsc{Ritov, Y.} (1990).
\newblock Decision theoretic optimality of the cusum procedure.
\newblock \textit{The Annals of Statistics} 1464--1469.

\bibitem[{Siegmund(2013)}]{siegmund2013sequential}
\textsc{Siegmund, D.} (2013).
\newblock \textit{Sequential analysis: tests and confidence intervals}.
\newblock Springer Science \& Business Media.

\bibitem[{Tartakovsky et~al.(2014)Tartakovsky, Nikiforov and
  Basseville}]{tartakovsky2014sequential}
\textsc{Tartakovsky, A.}, \textsc{Nikiforov, I.} and \textsc{Basseville, M.}
  (2014).
\newblock \textit{Sequential analysis: Hypothesis testing and changepoint
  detection}.
\newblock Chapman and Hall/CRC.

\bibitem[{Wald(1945)}]{Wald1945}
\textsc{Wald, A.} (1945).
\newblock Sequential tests of statistical hypotheses.
\newblock \textit{Ann. Math. Statist.}, \textbf{16} 117--186.

\bibitem[{Wang et~al.(2020)Wang, Yu, Rinaldo et~al.}]{wang2018univariate}
\textsc{Wang, D.}, \textsc{Yu, Y.}, \textsc{Rinaldo, A.} \textsc{et~al.}
  (2020).
\newblock Univariate mean change point detection: Penalization, cusum and
  optimality.
\newblock \textit{Electronic Journal of Statistics}, \textbf{14} 1917--1961.

\end{thebibliography}

\end{document}